\documentclass[a4paper,11pt]{amsart}
\usepackage{amssymb,amsmath,amsthm}
\numberwithin{equation}{section}
\usepackage{natbib}
\usepackage{paralist}               
\usepackage[margin=1in]{geometry}
\usepackage[titletoc]{appendix}     
\usepackage[final]{graphicx}
\usepackage{epstopdf}
\usepackage{subcaption}
\usepackage{float}
\usepackage{mathrsfs}
\usepackage{setspace}
\usepackage[bookmarks,bookmarksnumbered,bookmarksopen,breaklinks,linktocpage]{hyperref}
\newcommand{\ds}{\displaystyle}
\newcommand{\M}{\mathcal{M}}

\newcommand{\D}{\mathcal{D}}

\newcommand{\B}{\mathcal{B}}

\newcommand{\dd}{\mathrm{d}}

\newcommand{\sgn}{\operatorname{sgn}}
\newtheorem{Theorem}{Theorem}[section]
\newtheorem{Proposition}{Proposition}[section]
\newtheorem{Lemma}{Lemma}[section]

\theoremstyle{definition}           
\newtheorem{defn}{Definition}[section]
\theoremstyle{remark}

\newtheorem{Corollary}{Corollary}[section]
\usepackage{xcolor}
\usepackage{color,soul}

\begin{document}
	\title{Adjoint-based gradient methods for inverse design in a multiple fragmentation model}
	\author[Arijit Das]{Arijit Das}
	
	\address[Arijit Das]{Department of Mathematics \\
		Thapar Institute of Engineering and Technology, Patiala-147004, Punjab, INDIA}
	
	\email[Arijit Das]{\href{mailto:arijit.das@thapar.edu}{arijit.das@thapar.edu}}
	

	\begin{abstract}
		We study an inverse design problem for the linear multiple fragmentation equation arising in particle dynamics. Our objective is to reconstruct an unknown initial size distribution that evolves, under a prescribed fragmentation law, into a desired size distribution at a specified final time. We first establish the existence of global mass-conserving solutions for a broad class of fragmentation kernels with unbounded rates, and subsequently prove the continuous dependence and uniqueness of these solutions under additional assumptions on the fragmentation kernels. We then formulate the inverse design problem as an optimal control problem constrained by the fragmentation dynamics and prove the existence of the optimal control problem. Also derive the corresponding continuous adjoint equation and propose a gradient-type iterative reconstruction method. For the numerical implementation, we develop finite volume schemes for both the forward and adjoint equations, including a weighted finite volume scheme designed to enhance mass conservation and accuracy. Two benchmark problems, involving linear and nonlinear fragmentation rates with known analytical solutions, are used to assess the accuracy and efficiency of the proposed approach and to compare the performance of the two discretizations in both forward simulations and inverse reconstructions.
	\end{abstract}
	
	\maketitle
	\section{Introduction}
	Size distribution is a fundamental characteristic in many systems involving dispersed objects, such as droplets in lean liquid–liquid dispersions \cite{hu1995evolution}, colloids in solid–liquid dispersions \cite{harada2006dependence}, solid bodies in granulation processes \cite{fries2013collision}, and molecular species in polymerization or degradation mechanisms \cite{montroll1940theory}. Although these objects differ in their physical nature, the evolution of their size distribution is typically governed by population-level mechanisms, most notably coagulation (aggregation) and fragmentation (breakage). In this study, we focus exclusively on the fragmentation mechanism, in which larger entities divide to form smaller-sized particles.
	
	To present the multiple fragmentation equation, we consider \( f(x,t) \) to be the particle number density function describing the concentration of particles of size \( x \ge 0 \) at time \( t \ge 0 \). Then the linear multiple fragmentation is given by
		\begin{align}\label{0_1}
		\frac{\partial f(x,t)}{\partial t} =\mathcal{A}\left(f\right)(x,t), \quad \text{with}\quad  x \in \mathbb{R}_+.
	\end{align}
	Where
	\begin{align}\label{0_6}
		\mathcal{A}\left(f\right)(x,t) :=\int_{x}^{\infty} b \left(x,y\right)S(y)f(y,t) \dd y - S(x) f(x,t),
	\end{align}
	together with the initial data
	\begin{align}\label{0_2}
		f(x,0) = f_0(x), \quad x \in \mathbb{R}_+.
	\end{align}
		Here, $S(x)$ denotes the overall fragmentation rate of a particle of size $x$, and $b(x,y)$ represents the breakage distribution function. The function $b(x,y)$ characterizes the rate at which particles of size $x$ are produced from the fragmentation of a parent particle of size $y$. In addition, the function $b(x,y)$ satisfies the following conditions:
	\begin{align}
		&\int_{0}^{y} b(x,y) \dd x = \nu(y),\qquad\text{with}\qquad b(x,y) =0, \quad\text{for}\quad x\ge y;\label{0_3}\\
		& \int_{0}^{y} xb(x,y) \dd x = y. \label{0_4}
	\end{align}
	The condition~\eqref{0_3} represents the average number of fragments produced in a single fragmentation event; thus, for multiple-fragment breakage, we have $\nu \ge 2$. Furthermore, relation~\eqref{0_4} expresses the mass–conservation principle associated with an individual fragmentation event.

	Following the pioneering work of Ziff and McGrady \cite{ziff1985kinetics} on the fragmentation equation \eqref{0_1}, the study of fragmentation dynamics has attracted substantial interest across physics, mathematics, and engineering. The inherently complex nature of the fragmentation equation makes it challenging to obtain analytical solutions for broad classes of fragmentation kernels. Consequently, after the foundational contributions of Ziff and McGrady, further analytical results were established only for a few simplified kernel structures \cite{ziff1991new, dubovskii1992exact, singh1996kinetics}. 
	
	Subsequently, the existence of weak solutions to the binary fragmentation model with specific fragmentation rates was studied in several works \cite{stewart1989global, dubovskiǐ1996existence, camejo2015singular}. These studies relied on standard weak-compactness technique and required relatively strong assumptions on the aggregation rates. In contrast, the present work employs a strong convergence approach to establish the existence of solutions to the multiple-fragmentation equation under weaker conditions on the fragmentation kernels. From an applications perspective, analytical studies alone are insufficient for complex and practically relevant kinetic kernels. This limitation has motivated the development of a variety of computational techniques for approximating solutions, including sectional methods \cite{kumar2008convergence, kumar2015development, saha2023rate}, stochastic approaches \cite{mishra2000monte, das2020approximate}, and moment-based methods \cite{diemer2002moment, marchisio2003quadrature}.

	To summarize the above discussion, it is evident that most existing studies focus on the initial value problem, beginning with a prescribed initial distribution. However, in many applied contexts \cite{xue2013imaging, kostoglou2005self, doumic2023moments, doumic2018estimating}, estimating either the initial distribution or the fragmentation kernel itself is of greater importance than solving the forward evolution equation. This naturally leads to a fundamental research question: \emph{Can we uniquely recover the past (or present) distribution from the present (or future) state?} In other words, our aim is to determine whether the underlying dynamical system is time-reversible.
	
	 From a physical standpoint, such reversibility is difficult to expect. Once a particle undergoes fragmentation, information about its original structure is irretrievably lost, and many distinct parent particles may lead to identical fragment distributions. Moreover, fragmentation dynamics tend to smooth the distribution over time—analogous to diffusion. Solving the system backward in time therefore requires reconstructing a less regular state from a more regular one, an inherently unstable procedure. 
	 	
	In recent years, several studies \cite{alomari2013recovery, doumic2018estimating} have addressed the problem of estimating the fragmentation rate. Most of these works rely on parameter-identification techniques in which the parameters of the fragmentation model are determined by minimizing a least-squares functional that measures the discrepancy between experimental observations and model predictions, following an appropriate parametrization of the problem.\emph{ As discussed earlier, predicting the initial state from the final state is far from straightforward. Motivated by this challenge, the present work proposes an alternative framework to study the system in a time-reversible context. In particular, we introduce an \emph{optimal control strategy} aimed at approximating the initial state and achieving a desired final state for this class of particulate processes.}
	
	At this stage, we can formulate our inverse design problem as follows: for a prescribed target function $f^\ast = f^\ast(x)$ at $t=T$, we want to determine a initial data $f_0$ such that $\ds f(x,T) = f^\ast(x)$ for all $x\in \mathbb{R}_+$. To address this, we introduce the following functional, which represents the classical quadratic error between the solution at the final time and the target function $f^\ast$:
	
	\begin{align}\label{0_}
		\min J\left(f_0\right) = \frac{1}{2} \int_{\mathbb{R}_+} \left(f(x,T; f_0) - f^{\ast}(x)\right)^2 \dd x,
	\end{align}

	\begin{figure}
		\centering
	\includegraphics[width=1.0\textwidth]{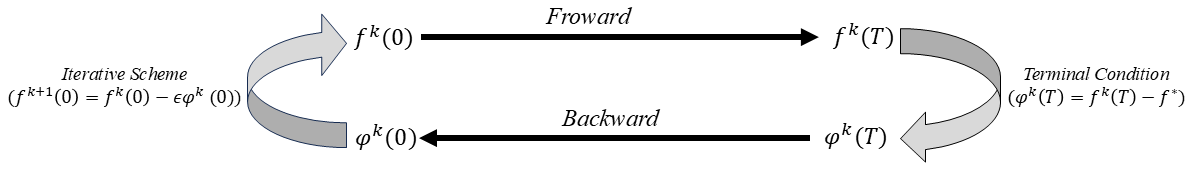}
	\caption{Flowchart of adjoint based method}
	\label{f0} 
	\end{figure}
	
	To solve the above inverse problem, we employ a gradient–adjoint iterative method. In this context, it is necessary to derive the corresponding adjoint equation. The adjoint equation for the adjoint variable $\varphi=\varphi(x,t)$ is given by
	\begin{align}\label{adj}
		\frac{\partial \varphi(x,t)}{\partial t} = -\int_{0}^{x} b (y,x) S(x) \varphi(y)\dd y + S(x) \varphi(x).	
	\end{align}
	with the terminal condition $\ds \varphi(T) = f(T)-f^\ast$. A detailed derivation of equation~\eqref{adj} will be presented in the following section. The adjoint-based gradient method is implemented through an iterative procedure in which the forward evolution equation is solved forward in time, while the corresponding adjoint equation is solved backward in time from \( T \) to \( 0 \). A schematic representation of the overall algorithmic workflow is provided in Figure~\ref{f0}.
	
	Over the past few decades, adjoint-based methods have gained significant attention across various fields of application, particularly in aeronautical engineering \cite{jameson1988aerodynamic, castro2007systematic} and oceanography \cite{power2006adjoint}. From a mathematical perspective, the adjoint equation provides an efficient means of computing the gradient of a cost functional with respect to model parameters or control variables, thereby enabling sensitivity analysis and the solution of optimization problems involving complex dynamical systems. Two classical approaches exist for formulating the adjoint equation backward in time: the continuous and the discrete adjoint formulations. In the continuous approach, one derives the adjoint equation analytically from the continuous forward model, subsequently discretizes it, and then solves it numerically. In contrast, the discrete approach obtains the adjoint equation directly by algebraic differentiation of the discretized forward scheme.
	
	To obtain an accurate approximation of the continuous solution, we implement higher-order numerical schemes for solving the forward fragmentation equation. In particular, we perform a comparative analysis between a second-order finite volume scheme and the weighted finite volume scheme introduced in \cite{kumar2015development}. Alongside developing the adjoint-based gradient algorithm for the inverse design problem in linear fragmentation, one of our objectives is to evaluate the effectiveness of employing numerical schemes of the same order in terms of accuracy. Our observations indicate that, in certain cases, the use of the weighted scheme—though capable of producing a more accurate final-time solution—may lead to substantially higher computational cost. However, this behavior is not universal, and in many situations it is possible to relax the requirement of adopting a higher-order scheme for the forward problem while still achieving an accurate reconstruction of the initial state.
	
\subsection{Contributions and organization of the work}
The structure of the present article is as follows:
\begin{itemize}
	\item After introducing the linear fragmentation model together with the associated inverse design problem, we establish in Section~\ref{S_2} the existence of solutions to the forward model for a general class of kinetic rates. By deriving a priori estimates for higher-order moments and applying the sufficient conditions of the Arzelà–Ascoli theorem, we obtain the main existence result. The section concludes with proofs of continuous dependence and uniqueness under additional assumptions on the fragmentation rates.
	
	\item In section~\ref{S_3} we define the optimal control problem to solve the reversibility of considered fragmentation model. The convergence of the solutions and continuous dependence on the initial data helps us to prove the existence of the optimal control problem. To implement the adjoint-based gradient algorithm, we derive the continuous adjoint equation together with the corresponding gradient-based iterative scheme. After presenting the adjoint formulation, we introduce the main iterative algorithm based on the gradient descent method.
	
	\item Section~\ref{S_4} presents the discrete formulations of both the forward and backward (adjoint) equations. 
	In particular, we implement two numerical schemes—the finite volume scheme (FVS) and the weighted finite volume scheme (WFVS)—to solve the forward problem. In addition, we newly develop a time-reversible numerical scheme for the derived adjoint equation.
	
	\item Finally, Section~\ref{S_5} provides a thorough numerical validation of the proposed algorithm. 
	We verify the consistency of the adjoint-based gradient by performing a Taylor test for gradient verification. 
	Furthermore, we assess the efficiency of the proposed algorithm through two benchmark tests: 
	one corresponding to a linear fragmentation rate and the other to a nonlinear rate.
\end{itemize}

	\section{Existence of the forward problem}\label{S_2}
	To proceed toward establishing the existence of solutions for unbounded kinetic rates, we begin by considering the case in which the fragmentation kernel has compact support. 
	For this purpose, we denote the collection of all continuous function in the weighted space $\ds L^1\big([0,\infty); x^{r}\,\mathrm{d}x\big)$ by $X_r$. For \( r \in \mathbb{R} \), the space $X_r$ 
	endowed with the norm
	\[
	\|f\|_r := \int_{0}^{\infty} x^r \left|f(x) \right|\dd x,
	\]
	In addition, to analyze the kinetic equation under consideration, 
	we define the moment functions of order \( p \) as follows:
	\begin{align*}
		\mathcal{M}_p (t) := \int_{0}^{\infty} x^p f(x,t) \dd x.
	\end{align*}
	Among these moments the zeroth and first order moments describe the total number and mass of the particles present in the system. Therefore it is natural to work in the intersecting space defined as
	\begin{align*}
		X_{0,r} := X_0 \cap X_r =  L^1\big([0,\infty); \left(1+x^{r}\right)\,\mathrm{d}x\big) \quad \text{for some}\quad r\ge 1.
	\end{align*}
	$X_{0,r}$ form a Banach space with respect to the norm 
	\begin{align*}
	\|f\|_{0,r} := \int_{0}^{\infty} \left(1+x^r\right) \left|f(x) \right|\dd x	,
	\end{align*}
	and let \( X_{0,r}^{+} \) denote the positive cone of \( X_{0,r} \). 
	
	\begin{defn}
		Let $T\in \left(0,\infty \right)$. A solution of the problem \eqref{0_1}-\eqref{0_2} is a non negative function $f: [0,T]\to X_{0,r}^+$ for some $r\ge 1$, such that for $x\in (0,\infty)$ a.e.
		\begin{enumerate}
			\item[i.] $s \mapsto f(x,s)$ is continuous on $[0,T]$,
			\item[ii.] for all $t \in [0,T]$, $\ds \int_{x}^{\infty} b(x,y)S(y) f(y) \dd y \in L^1 \left(\left[0,T\right]\right)$, and
			\item[iii.] for all $t \in [0,T]$,
			\[
			f(x,t) = f_0(x) + \int_{0}^{t} \left(\int_{x}^{\infty} b \left(x,y\right)S(y)f(y,s) \dd y - S(x) f(x,s)\right)\dd s.
			\]
		\end{enumerate}  
	\end{defn}
	
	We prove the existence of strong solutions for the multiple fragmentation equation \eqref{0_1} under the following assumptions on the fragmentation rate $S$ and distribution function $b$
	\begin{enumerate}
		\item[$H_1$.] $S(x) \le S_0 x^\alpha$ with $S_0 >0$ and $0 \le \alpha\le  r-1$.
		\item[$H_2$.] There exist a constant $\bar{b} >0$ such that
		\begin{align*}
			xb(x,y) \le \bar{b}\quad \text{for all}\quad 0\le x \le y <\infty.
		\end{align*}
	\end{enumerate}
	
	 First, we need to show that a unique solution to the IVP \eqref{0_1}–\eqref{0_2} exists when the fragmentation rate \( S(x) \) has compact support. In view of this, we state and prove the following lemma.
	\begin{Lemma}[Existence theorem for compactly supported kernel]\label{lemma_1}
		Let \( b \) and \( S \) be nonnegative functions, with \( b \) continuous on \( (0,\infty)^2 \) and \( S \) continuous on \( (0,\infty) \). Assume further that \( S \) has compact support and that \( b \) satisfies condition~\eqref{0_4}. If the initial data \( f_{0} \in X_{0,r}^{+} \) for some $r\ge 1$, then the initial value problem \eqref{0_1}--\eqref{0_2} admits at least one solution
		\(
		f \in \mathcal{C}\big([0,T]; X_{0,r}^{+}\big)
		\)
		for some \( T > 0 \). 
	\end{Lemma}
	The proof of the above lemma is standard, and the existence and uniqueness of solutions follow directly from Theorem~2.1 in \cite{das2022existence}, to which we refer for further details.
	
	Now we construct a sequence of fragmentation kernels \(\{S_n\}_n\) with the help of above compactly supported kernel $S$.
%
	\[
	S_n(x) = \begin{cases} S(x), & 0\le x\le n,\\ \text{smoothly decreasing to }0, & x\in[n,n+1],\\ 0, & x\ge n+1. \end{cases}
	\]
	From the above construction, it is evident that \( S_n \) increases monotonically to \( S \) and has compact support. For this \emph{cutoff kernel}, the existence of continuous, non-negative solutions \( f_n \) for each \( n \) on the compact domain is guaranteed by Lemma~\ref{lemma_1}. In this setting, the modified problem takes the following form:
	
	\begin{align}\label{1_1}
		\frac{\partial f_n(x,t)}{\partial t} =\int_{x}^{\infty} b(x,y)S_n(y)f_n(y,t) \dd y - S_n(x) f_n(x,t), ~ \text{with} ~ x \in \mathbb{R}_+.
	\end{align}
	with the initial data
	\[f_n(x,0) = \begin{cases} f_0(x), & 0\le x\le n,\\ 0, & x>n. \end{cases}\]
	
	Let us denote the $p-$th order moment corresponding to the solution $f_n$ as
	\begin{align*}
		\mathcal{M}_{p,n} (t) := \int_{0}^{\infty} x^p f_n(x,t) \dd x.
	\end{align*}
	
	\subsection{Prior estimation of the moments}
	Consider \( f(x,t) \) to be a solution of the forward equation \eqref{1_1} in 
	\(\mathcal{C}\!\left([0,T]; X_{0,r}^+\right)\) for some fixed \( T>0 \) and \( r\ge 1\). 
	To estimate the moment functions, we rewrite equation \eqref{1_1} in its moment form. 
	To this end, we multiply \eqref{1_1} by a test function \( \phi(x) \) and integrate with respect to \( x \).
	\begin{align}\label{mom_eq}
		\partial_t \int_{0}^{\infty} \phi(x) f_n(x,t) \dd x = \int_{0}^{\infty} \psi_\phi(y)  S_n(y) f_n(y,t) \dd y 
	\end{align}
	where $\ds \psi_\phi(y) := \int_{0}^{y} \phi(x) b(x,y) \dd x -\phi(y)$. 
	
	\begin{Proposition}\label{prop_1}
		Let $f_{0} \in X_{0,r}^+$ for some $r\ge 1$ and consider $f_n$ be the solution to equation \eqref{1_1} on $[0,T]$ such that
		\begin{align}
			\int_{0}^{\infty}xS_n(x)f_n(x,s)\dd x\in L^1([0,T]) .\label{2_15}
		\end{align}
		Then the solution is mass conserving, that is
		\begin{align*}
			\mathcal{M}_{1,n}(t)=\mathcal{M}_{1}(0):=\bar{M}_1 \quad \text{for all}\quad t \in [0,T].
		\end{align*}
	\end{Proposition}
	
	\begin{proof}
		Let $M$ be positive integer. We consider $\ds \phi_M(x) := x \chi_{(0,M)}(x)$ and apply the relation  \ref{mom_eq} whenever $\ds \phi_M(x) \in L^\infty(0,\infty)$ to get
		\begin{align}
			\int_{0}^{\infty} \phi_M(x) f_n(x,t)\dd x = \int_{0}^{\infty} \phi_M(x) f_0(x)\dd x+\int_{0}^{t}\int_{0}^{\infty}\psi_{\phi_M}(y) S_n(y) f_n(y,t)\dd y  \dd s .\label{prop_1.1}
		\end{align}
		
		Now the following cases arise;
		
		\begin{itemize}
			\item[a)] $y \le M$, $\implies$ $\ds \psi_{\phi_M}(y)= \int_{0}^{y}xb(x,y)\dd x-y=0$,
			\item[b)] $y > M$,  $\implies$ $\ds  \psi_{\phi_M}(y,z)=y-\int_{M}^{y}xb(x,y)\dd x \in (0,y)$.
		\end{itemize}
		
		Combining all the cases, we get
		
		\begin{align*}
			0\le \psi_{\phi_M}(y) \le y \quad\text{for all}\quad (y,z)\in (0,\infty)^2.
		\end{align*}
		
		Using the above inequality on equation \eqref{prop_1.1}
		
		\begin{align}\label{prop_1.2}
			\int_{0}^{\infty} \phi_M(x) f_n(x,t)\le  \int_{0}^{\infty} \phi_M(x) f_0(x)+ J^*_1\le \M_1(0)+J^*_1 \quad\text{for all}\quad t\in [0,T],
		\end{align}
		with $\ds J^*_1 := \int_{0}^{t}\int_{0}^{\infty}yS(y)f_n(y,s)\dd y  \dd s$.
		
		From equation \eqref{2_15}, it is clear that $J^*_1$ is finite and since $f_{0} \in X_1^+$, the right hand side of the inequality \eqref{prop_1.2} is finite and does not dependent on $M$. Letting $\ds M \longrightarrow \infty$ on \eqref{prop_1.2} and considering Fatou's theorem, we get $ \M_{1}(t)<\infty$, with $\ds \M_{1}(t) \le \M_1(0)+J^*_1$ for all $t \in [0,T]$. 
		
		Additionally, the following results
		
		\begin{align*}
			&\lim\limits_{M\to \infty} \psi_{\phi_M}(y)=0 \quad\text{for all}\quad y\in (0,\infty),\quad\text{and}\quad\\
			&\psi_{\phi_M}(y)S_n(y)f_n(y,s)\le yS_n(y)f_n(y,s)\quad\text{for all}\quad(y,s)\in (0,\infty)\times(0,t),
		\end{align*}
		together with equation \eqref{2_15}, the Lebesgue dominated convergence theorem gives
		
		\begin{align*}
			\lim\limits_{M\to \infty}\int_{0}^{t}\int_{0}^{\infty}\int_{0}^{\infty}\psi_{\phi_M}(y) S_n(y)f_n(y,s)\dd y \dd s=0.
		\end{align*}
		
		Consequently, with the help of relation \ref{mom_eq} we conclude that
		
		\begin{align}
			\M_{1,n}(t)= \lim\limits_{M \to \infty} \int_{0}^{\infty} \phi_M(x) f_n(x,t)\dd x=\lim\limits_{M\to \infty} \int_{0}^{\infty} \phi_M(x) f_0(x)\dd x=\M_1(0)=\bar{M}_1.\label{2_16}
		\end{align}
	\end{proof}
	
	We now prove the uniform boundedness of all moments for $p\ge 2$.
	
	\begin{Proposition}\label{prop_2}
		Let $f_{0} \in X_{0,r}^+$ for some $r\ge 1$ and consider the mass conserving solution $f_n$ to equation \eqref{1_1} on $[0,T]$. If the breakage distribution function $b$ satisfies the mass conserving condition \eqref{0_4}, then
		
		\begin{align*}
		\M_{i,n}(t) \le \M_i(0):=\bar{M}_i\quad\text{for all}\quad i\ge 2.
		\end{align*}
	\end{Proposition}
	
	\begin{proof}
		Consider $M>0$ and the relation \ref{mom_eq} for $\ds \phi_M (x) := (x^i-M^{i-1}x) \chi_{(0,M)}(x)$ where $x \in (0,\infty)$. Then equation \eqref{0_1} leads to the following cases;
		
		\begin{itemize}
			\item[i)] when $y\in (0,M)$,
			
			\begin{align*}
				\ds \psi_{\phi_M}(y) & = \int_{0}^{y}(x^i-M^{i-1}x)b(x,y)\dd x-(y^i-M^{i-1}y)\\
				& \le (y^{i-1}-M^{i-1})\int_{0}^{y}xb(x,y)\dd x  -(y^i-M^{i-1}y)=0,
			\end{align*}
			
			\item[ii)] when $y\in (M,\infty)$, $\ds  \psi_{\phi_M}(y) \le (M^{i-1}-M^{i-1})\int_{0}^{M}xb(x,y)\dd x  =0$.
		\end{itemize}
		
		Since $\phi \in L^\infty(0,\infty)$, then combining all cases we can conclude from the relation \eqref{mom_eq} that
		
		\begin{align*}
			\int_{0}^{M}(x^i-M^{i-1}x)\left[f_n(x,t)-f_{0}(x)\right]\dd x \le 0 \quad\text{for all}\quad t \in [0,T].
		\end{align*}
		
		Using the mass conserving property of $f_n$ on the above inequality, we get
		
		\begin{align*}
			\int_{0}^{\infty}\min\{x^i,M^{i-1}x\}f(x,t)\dd x \le \int_{0}^{\infty}\min\{x^i,M^{i-1}x\}f_0(x)\dd x \le \M_i(0)\quad\text{for all}\quad t \in [0,T].
		\end{align*}
		
		In particular, for all $M>0$
		
		\begin{align*}
			\int_{0}^{M} x^i f_n(x,t)\dd x \le \M_i(0).
		\end{align*}
		
		Passing $M \longrightarrow \infty$ on the above inequality, we can complete the proof.
	\end{proof}
	Using the Proposition  standard change of integration and the fact that $\ds f_0 \in X_{0}^+$, we can also estimate the zeroth order moment as
	\begin{align}
		\mathcal{M}_{0,n} (t) \le \bar{M}_0\quad\text{for all}\quad t \in [0,T].
	\end{align}
Using the previously obtained moment estimates, we now establish the boundedness and equicontinuity of the solution sequence on compact subsets with respect to the uniform topology. In this direction, we state and prove the following lemma.

\begin{Lemma}\label{lemaa_2}
		Let $T>0$, and suppose that the kernels $S(x)$ and $b(x,y)$ satisfy assumption $(H_1)$ and $(H_2)$, respectively. If the initial data $f_0(x)\in X_{0,r}^{+}$ for some $r\ge 1$, then the sequence of solutions \( \{f_n\} \) is relatively compact in the space of continuous functions endowed with the uniform topology on any compact rectangular 
	subset of 
	\[
	\Xi := \{(x,t):\, 0 < x < \infty,\; 0 \le t \le T\}.
	\]
\end{Lemma}
	
	\begin{proof}
		The proof of this lemma proceeds in the following steps:
		\begin{enumerate}
			\item First, we establish the uniform boundedness of the sequence \( \{f_n\} \) on a compact subset of \( \Xi \).
			\item Next, we prove that the sequence \( \{f_n\} \) is equicontinuous with respect to both the volume variable \( x \) and the time variable \( t \) on the same compact subset of \( \Xi \).
		\end{enumerate}
		For this purpose, we introduce a family of compact rectangular subsets of \( \Xi \) defined by
		
		\begin{align*}
			\Xi_1 :=  \{(x,t):\frac{1}{X}\le x \le X,~ 0\le t \le T\} 
		\end{align*}
		Where $X$ $(\gg 1)$ is finite number.
		\subsection{Uniform boundedness:}
		Recalling \eqref{1_1} and $\alpha\le 1$, for $(x,t) \in \Xi_1$
		\begin{align}\label{1_6}
			\frac{\partial f_n(x,t)}{\partial t}& \le \int_{x}^{\infty} b(x,y)S_n(y)f_n(y,t) \dd y \notag \\
			& = \bar{b} \int_{1/X}^{\infty} \frac{1}{x}S_n(y)f_n(y,t) \dd y \notag\\
			& \le XS_0 \bar{b} \int_{1/X}^{\infty} y^\alpha f_n(y,t) \dd y \notag \\
			& \le  X S_0 \bar{b} \left(\bar{M}_0 + \bar{M}_{\lceil{\alpha}\rceil}\right).
		\end{align}
		Integrating the above inequality and set $\bar{f}_0 := \max_{1/X \le x \le X} f_0(x)$, we have
		\begin{align}\label{bdd}
			f_n(x,t) \le \bar{f}_0+ X S_0 \bar{b}T \left(\bar{M}_0 + \bar{M}_{\lceil{\alpha}\rceil}\right) := \Gamma_1
		\end{align}
		Since, $\Gamma_1$ in independent of $n$ and $t$, so the solution $\{f_n\}$ is uniformly bounded on $\Xi_1$.
		
		\subsection{Time equicontinuity}
		Let \(0 \le t \le t' \le T\) and \(\epsilon > 0\).  
		To establish the equicontinuity of \(f_n\) with respect to \(t\) on each set \(\Xi_1\), we must show that for every arbitrary \(\epsilon > 0\), there exists a \(\delta(\epsilon) > 0\) such that
		
		\begin{align*}
			|t'-t| < \delta(\epsilon) \quad \mbox{implies} \quad \left|f_n(x,t') - f_n(x,t)\right| < \epsilon, \quad \text{for all}\quad \frac{1}{X}\le x \le X.
		\end{align*}
		
		Now from equation \eqref{1_1}, we write for each $n\ge 1$
		
		\begin{align*}
			|f_n(x,t') - f_n(x,t)| \le \int_t^{t'} & \left| \int_{x}^{\infty} b(x,y)S_n(y)f_n(y,s) \dd y - S_n(x) f_n(x,s)\right| \dd s.
		\end{align*}
		Recalling the uniform boundedness of $\{f_n\}$ and the estimation \eqref{1_6}, the above inequality further estimated as:
		\begin{align*}
			|f_n(x,t') - f_n(x,t)| \le \int_t^{t'}\left[ X S_0 \bar{b} \left(\bar{M}_0 + \bar{M}_{\lceil{\alpha}\rceil}\right) + S_0 X^\alpha \Gamma_1 \right]\dd s
		\end{align*}
		Consider the constant $\ds \zeta_1:= X S_0 \bar{b} \left(\bar{M}_0 + \bar{M}_{\lceil{\alpha}\rceil}\right) + S_0 X^\alpha \Gamma_1$. Therefore, for $\epsilon >0$ there exist a $\delta(\epsilon)$ such that for $n\ge 1$;
		\begin{align}\label{1_10}
			|f_n(x,t')-f_n(x,t)| \leq \zeta_1 \epsilon \quad \text{whenever} \quad |t'-t| < \delta(\epsilon). 
		\end{align}
		
	\subsection{Space equicontinuity}
	
	Let \(\frac{1}{X} \le x < x' \le X\).  
	To establish the equicontinuity of \(f_n\) with respect to the spatial variable \(x\), we need to show that for every arbitrary \(\epsilon > 0\), there exists a \(\delta(\epsilon) \le \epsilon\) such that
	
	\[
	|x' - x| < \delta(\epsilon)
	\quad \Longrightarrow \quad
	\left| f_n(x',t) - f_n(x,t) \right| < \epsilon,
	\qquad \text{for all } 0 \le t \le T.
	\]
	
	By construction, each of the kernels \(S_n\) is continuous on the closed interval \([0,X]\).  
	We also consider a bounded domain \(0 < Z_1 \le y \le Z_2 < \infty\), where the values of \(Z_1\) and \(Z_2\) will be specified in the subsequent analysis.  
	
	Furthermore, the fragmentation kernel \(b\) is continuous on the rectangle 
	\([\tfrac{1}{X}, X] \times [Z_1, Z_2]\).  
	Using the continuity of the initial data and the kinetic kernels, we obtain the following relations whenever \(0 < \delta(\epsilon) < \epsilon\):
	
		\begin{align*}
			\sup_{|x'-x|<\delta} |f_0(x')-f_0(x)|<\epsilon,
		\end{align*}
		and
		\begin{align*}
			\sup_{|x'-x|<\delta} \left|S_n(x')-S_n(x)\right|<\epsilon, \quad \sup_{|x'-x|<\delta} \left|b(x',y)-b(x,y)\right| < \epsilon.
		\end{align*}
		Now for each $n\ge 1$, we can write
		\begin{align}\label{1_7}
			\begin{split}
				\left|f_n(x',t)-f_n(x,t)\right| &  \le  \left|f_0(x')- f_0(x)\right| + \int_{0}^{t}\left[\int_{x}^{\infty} \left|b(x',y) - b(x,y)\right|S_n(y) f_n(y,s)\dd y\right.\\
				& \left. + \int_{x}^{x'}b(x',y)S_n(y) f_n(y,s) \dd y + \left|S_n(x') - S_n(x)\right|f_n(x',s)\right.\\
				&\left. + S_n(x)\left|f_n(x',s) - f_n(x,s)\right|\right]\dd s.
			\end{split}
		\end{align}
		Define
		\begin{align*}
			\omega_n(t) := \sup_{|x'-x|<\delta} \left|f_n(x',t)-f_n(x,t)\right|, \quad \frac{1}{X}\le x,x' \le X.
		\end{align*}
		We first estimate the first integral of the right hand side
		\begin{align*}
			\int_{x}^{\infty} \left|b(x',y) - b(x,y)\right|S_n(y) f_n(y,t)\dd y \le& \int_{x}^{Z_1} \left|b(x',y) - b(x,y)\right|S_n(y) f_n(y,s)\dd y \\
			&+ \int_{Z_1}^{Z_2} \left|b(x',y) - b(x,y)\right|S_n(y) f_n(y,s)\dd y \\
			&+\int_{Z_2}^{\infty} \left|b(x',y) - b(x,y)\right|S_n(y) f_n(y,s)\dd y:= J_1 +J_2 +J_3.
		\end{align*}
		\begin{align*}
			J_1 \le 2\bar{b} \int_{x}^{Z_1} \frac{1}{x} S_n(y)f_n(y,s)\dd y \le 2\bar{b}  S_0  \bar{M}_1X Z_1^{\alpha-1}.
		\end{align*}
		\begin{align*}
			J_2 \le \epsilon \int_{Z_1}^{Z_2} S_n(y) f_n(y,s)\dd y \le S_0 \bar{M}_{\lceil{\alpha}\rceil} \epsilon.
		\end{align*}
		\begin{align*}
			J_3 \le 2 \bar{b}\int_{Z_2}^{\infty} \frac{1}{x} y^\alpha f_n(y,s)\dd y  \le  \frac{2\bar{b}S_0 X}{Z_2} \bar{M}_{\lceil{1+\alpha}\rceil}
		\end{align*}
		If we choose $Z_1$ and $Z_2$ such a way that, they satisfy the relation $\bar{M}_1 Z_1^{\alpha-1}\le \epsilon$ and $\frac{1}{Z_2} \bar{M}_{\lceil{1+\alpha}\rceil} \le \epsilon$ respectively. Then using these relation on the above estimations of $J_1$, $J_2$ and $J_3$ we get
		\begin{align}
			\int_{x}^{\infty} \left|b(x',y) - b(x,y)\right|S_n(y) f_n(y,t)\dd y \le \zeta_2 \epsilon,
		\end{align}
		where $\ds \zeta_2 := S_0 \left(4\bar{b}X+ \bar{M}_{\lceil{\alpha}\rceil}\right)$.
		
		The second integral of the right hand side \eqref{1_7} as follows
		\begin{align}
			 \int_{x}^{x'} b(x',y) S_n(y) f_n(y,s)\dd y \le \bar{b} S_0 X \int_{x}^{x'}y^\alpha f_n(y,s)\dd y \le\bar{b} S_0 X  \Gamma_1 \epsilon.
		\end{align}
		Using of these estimation together with the uniform boundedness of $\{f_n\}$ on the inequality \eqref{1_7} leads us to the following estimation
		\begin{align*}
			\omega_n(t) \le \zeta_3 \epsilon + \zeta_4 \int_{0}^{t}	\omega_n(s)\dd s,
		\end{align*}
		where $\zeta_3:= 1+\bar{M}_0+\zeta_2$ and $\zeta_4:= S_0X^\alpha$.
		 Therefore, applying Gronwall's inequality, we get
		
		\begin{align}
			\omega_n(t)\le \zeta_3\exp\left(\zeta_4 T\right)\epsilon. \label{1_9}
		\end{align}
		
		Hence, the equicontinuity with respect to $x$ on $\Xi_1$ is obtained.
		
		Thus, from the results \eqref{1_10} and \eqref{1_9}, we can conclude that
		
		\begin{align}
			\sup_{|x'-x| < \delta,~|t'-t|<\delta} \left|f_n(x',t') - f_n(x,t)\right|\le \left[ \zeta_3\exp\left(\zeta_4 T\right) + \zeta_1\right]\epsilon, \label{1_11}
		\end{align}
		
		whenever $ \frac{1}{X}\le x\le x' \le X$ and $0\le t\le t' \le T$. Furthermore, by combining the estimates in \eqref{bdd} and \eqref{1_11} with the Arzel\`a–Ascoli theorem, we conclude that there exists a subsequence \(\{f_{n_s}\}_{s=1}^\infty\) which is relatively compact in the topology of uniform convergence on each rectangle \(\Xi_1\).
	\end{proof}
	
	\subsection{Existence of global solutions:}
	\begin{Theorem}[Existence theorem]\label{th1}
		Let $T>0$, and suppose that the kernels $S(x)$ and $b(x,y)$ satisfy assumption $(H_1)$ and $(H_2)$, respectively. If the initial data $f_0(x)\in X_{0,r}^{+}$ for some $r\ge 1$,, then the IVP \eqref{0_1}–\eqref{0_2} admits at least one mass conserving solution 
		\(
		f \in \mathcal{C}\!\left([0,T]; X_{0,r}^{+}\right).
		\)
	\end{Theorem}

	\begin{proof}
		By applying the diagonal method, one can extract a uniformly convergent subsequence 
		$\{f_{p}\}_{p=1}^{\infty}$ from the sequence $\{f_{n}\}_{n=1}^{\infty}$ on each compact subset of $\Xi$.  
		This subsequence converges to a continuous, nonnegative limit function $f$ that satisfies \eqref{bdd}.  
		To proceed further, consider the integral
		\[
		\int_{\bar{z}_1}^{\bar{z}_2} \left(1+x^j\right)\, f(x,t)\, \dd x,\quad \text{for} \quad 0\le j \le r,
		\]
		where $0<\bar{z}_1<\bar{z}_2$.
		
		Now, the subsequence $\{f_p\}$ ensures that for all $\epsilon>0$ there exists $p\ge 1$ such that, the following relation holds good.
		\begin{align}
			\int_{\bar{z}_1}^{\bar{z}_2}  \left(1+x^j\right) \left|f(x,t)-f_p(x,t)\right| {\dd}x \le  \epsilon. \label{1_12}
		\end{align}
		
		Furthermore, $\epsilon$, $\bar{z}_1$ and $\bar{z}_2$ are arbitrary in \eqref{1_12}. Therefore, for all $0\le j\le r$, we have
		\begin{align}
			\int_{0}^{\infty} \left(1+x^j\right)f(x,t)\, {\dd}x \le \bar{M}_{0} + \bar{M}_{r}. \label{1_13}
		\end{align}
		
		Our aim to show that $f(x,t)$ is a solution to the IVP \eqref{0_1}-\eqref{0_2}. In this regard, replace $S_n$, $f_n$ in \eqref{1_1} with $S_p-S+S$, $f_p-f+f$ respectively, and rearrange the terms to obtain:
		\begin{align}
			\begin{split}
				\left(f_p - f\right)(x,t) + f(x,t) &=  f_0(x)+ \int_0^t \left[ \int_{x}^{\infty} b(x,y) \left(S_p - S\right)(y) f_p(y,s)\dd y\right.\\
				& \left. +\int_{x}^{\infty} b(x,y)S(y) \left(f_p - f\right)(y,s)\dd y + \int_{x}^{\infty} b(x,y)S(y)f(y,s)\dd y \right.\\
				&\left.+ \left(S_p - S\right)(x) f_p(x,s) - S(x) \left(f_p - f\right)(x,s) - S(x)f(x,s)\right]{\dd}s. \label{1_14}
			\end{split}
		\end{align}
		Using previous arguments, we get
		\begin{align*}
			\int_{x}^{\infty} b(x,y) \left(S_p - S\right)(y) f_p(y,s)\dd y \le \zeta_5 \epsilon \quad \int_{x}^{\infty} b(x,y)S(y) \left(f_p - f\right)(y,s)\dd y\le \zeta_6 \epsilon
		\end{align*}
		for some constant $\zeta_5$ and $\zeta_6$. Using the definition of convergence for a sequence of functions, it follows directly that all integrals involving the terms 
		$\left(S_p - S\right)$ and $\left(f_p - f\right)$ vanish as $p \to \infty$.  
		Consequently, letting $p \to \infty$ in the identity \eqref{1_14} yields
		\begin{align}\label{1_16}
			f(x,t) 
			= f_0(x) 
			+ \int_{0}^{t}
			\left[
			\int_{x}^{\infty} b(x,y)\, S(y)\, f(y,s)\, \dd y 
			- S(x)\, f(x,s)
			\right]\dd s,
			\quad x \in \mathbb{R}_+.
		\end{align}
		
		The tail estimates together with the continuity of $f(x,t)$ imply that the right-hand side of \eqref{0_1}, when evaluated at $f$, defines a continuous function on $\Xi$.  
		Moreover, the relation \eqref{1_13} ensures that $f(x,t) \in X_{0,r}^+$, and differentiating \eqref{1_16} with respect to $t$ shows that $f(x,t)$ is a continuous and differentiable solution of the IVP \eqref{0_1}–\eqref{0_2}.
		
		\emph{Mass conservation:} To prove the the obtained solution is mass conserving, it is enough to show that the condition \eqref{2_15} holds good. Observe that
		\begin{align*}
				\int_{0}^{T}\int_{0}^{\infty}xS_n(x)f_n(x,s) \dd x \dd s \le S_0 T \left[\bar{M}_0+\bar{M}_r\right]
		\end{align*}
		Thanks to the Proposition \ref{prop_1}, the mass conservation follows for the obtained solution.
	\end{proof}
	
	\subsection{Continuous Dependence on Initial Data and Uniqueness}
	\begin{Proposition}\label{prop_3}
		Assume that the fragmentation kernels satisfy the assumptions $H_1$ and $H_2$, and let $f$ and $\hat{f}$ be the solutions of \eqref{0_1} corresponding to the initial conditions $f(x,0)=f_0(x)\in X_{0, r}^{+}$ and $\hat{f}(x,0)=\hat{f}_0(x)\in X_{0, r}^{+}$, respectively, satisfying
		\begin{align}\label{ini_1}
			f,\, \hat{f} \in \mathcal{C}\!\left([0,T]; X_{0,r}^{+}\right) \quad \text{for}\quad r\ge 1,
		\end{align}
		for all $T>0$. Then, for each $T>0$, there exists a positive constant $\kappa$, depending on $T$, $\|f_0\|_{0,1}$, and $\|\hat{f}_0\|_{0,1}$, such that
		\[
		\sup_{t \in [0, T]} \|f - \hat{f}\|_{0,1} \le \kappa\, \|f_0 - \hat{f}_0\|_{0,1}.
		\]
	\end{Proposition}
	\begin{proof}
		Let us define the function 
		\[
		\phi(x,t) := f(x,t) - \hat{f}(x,t), \qquad \phi_0(x) := f(x,0) - \hat{f}(x,0),
		\]
		and consider the function $\sgn(h)$ given by
		\[
		\sgn(h)=
		\begin{cases}
			\dfrac{h}{|h|}, & \text{if } h \ne 0,\\[1ex]
			0, & \text{if } h = 0,
		\end{cases}
		\qquad\text{and}\qquad 
		\dfrac{d|w(h)|}{dh} = \sgn(w(h)) \dfrac{dw(h)}{dh}.
		\]
		
		Using the definition of $\phi$, equation \eqref{0_1} may be rewritten as
		\[
		\frac{\partial \phi}{\partial t}(x,t)
		= \int_{x}^{\infty} b(x,y) S(y)\, \phi(y,t)\, dy - S(x)\, \phi(x,t).
		\]
		Multiplying this equation by $(1+x) \,\sgn(\phi)$ and integrating with respect to $x$ over $(0,\infty)$, we obtain
		\[
		\|\phi(t)\|_1
		\le \|\phi_0\|_1 
		+ \int_{0}^{t} \int_{0}^{\infty} (1+x)\, \sgn(\phi(x,s)) 
		\left[ 
		\int_{x}^{\infty} b(x,y) S(y)\, \phi(y,s)\, dy 
		- S(x)\phi(x,s)
		\right] dx\, ds.
		\]
		This further implies
		\[
		\|\phi(t)\|_1 
		\le \|\phi_0\|_1
		+ \int_{0}^{t} \left[
		\int_{0}^{\infty}\!\!\int_{x}^{\infty} 
		(1+x)\, b(x,y) S(y)\, |\phi(y,s)|\, dy\, dx
		- \int_{0}^{\infty} (1+x) S(x)\, |\phi(x,s)|\, dx
		\right] ds.
		\]
		By changing the order of integration in the first term and using conditions \eqref{0_4} and \eqref{ini_1}, we obtain the desired stability estimate. 
		\end{proof}
		
		\begin{Corollary}
			Assume that all hypotheses of Theorem~\ref{th1} hold for some $r \ge 1$. If the initial datum satisfies $f_0 \in X_{0,r}^+$, then the IVP \eqref{0_1}–\eqref{0_2} admits a unique mass conserving solution
			\[
			f \in \mathcal{C}\!\left([0,T]; X_{0,r}^{+}\right),
			\qquad \text{for each }\, 0 < T < \infty.
			\]
		\end{Corollary}
		
		\begin{proof}
			Uniqueness follows immediately from Theorem~\ref{th1} together with Proposition~\ref{prop_3}.
		\end{proof}

	\section{Derivation of adjoint equation and gradient descent method}\label{S_3} 
	

In this section we will study the optimal control problem to solve the reversibility of the multiple fragmentation. We first define the admissible control as follows
\[
\mathcal{U}_{\mathrm{ad}}
:=\Big\{ f_0 \in X_{0,r}^+ \cap L^2(\mathbb{R}_+) \;:\;
0 \le f_0(x) \le M \ \text{a.e. in } \mathbb{R}_+ \Big\},
\]
for some constant $M>0$. For some specific target number density function $f^\ast$, we consider the cost functional $\ds J: \mathcal{U}_{\mathrm{ad}} \to \mathbb{R}$, defined as
\begin{align}\label{Cost_Fun}
	J(f_0)
	:=\frac12 \int_{\mathbb{R}_+}
	\big(f(x,T;f_0)-f^\ast(x)\big)^2\,dx,
\end{align}
	subject to the constraint
\begin{align}
	& \partial_t f (x,t) = \mathcal{A}\left(f\right)(x,t), \quad \text{where}\quad \left(x,t\right) \in \mathbb{R}_+ \times \left[0,T\right],\\
	& f(x, 0)  =  f_0 (x)\ge 0 \quad \text{where}\quad x \in \mathbb{R}_+.
\end{align}

	\begin{Theorem}[Existence of a minimizer] 
		Assume that the fragmentation kernel \(S\)and the distribution function $b$ satisfies hypotheses {\rm($H_1$)} and {\rm($H_2$)}, respectively. Then the optimal control problem
		\[
		\min_{f_0 \in \mathcal{U}_{\mathrm{ad}}} J(f_0)
		\]
		admits at least one minimizer $\bar f_0 \in \mathcal{U}_{\mathrm{ad}}$.
	\end{Theorem}
	\begin{proof}
		From the definition \eqref{Cost_Fun}, $J(f_0)\ge 0$ for all $f_0\in\mathcal{U}_{\mathrm{ad}}$. So, there exists a minimizing sequence
		$\{f_0^n\}\subset\mathcal{U}_{\mathrm{ad}}$ such that
		\[
		J(f_0^n)\to \inf_{f_0\in\mathcal{U}_{\mathrm{ad}}} J(f_0).
		\]
		
		By definition of $\mathcal{U}_{\mathrm{ad}}$, the sequence $\{f_0^n\}$ is bounded in
		$L^2(\mathbb{R}_+)$. Thanks to the reflexivity of $L^2(\mathbb{R}_+)$, there exist a subsequence
		(relabeling as $\{f_0^n\}$) and a function
		$\bar f_0 \in L^2(\mathbb{R}_+)$ such that
		\[
		f_0^n \rightharpoonup \bar f_0
		\quad \text{weakly in } L^2(\mathbb{R}_+).
		\]
		
		The admissible set $\mathcal{U}_{\mathrm{ad}}$ is convex and closed in $L^2(\mathbb{R}_+)$;
		hence it is weakly closed. Therefore, $\bar f_0\in\mathcal{U}_{\mathrm{ad}}$.
		
		Let $f^n(\cdot,t)$ and $\bar f(\cdot,t)$ denote the solutions of the fragmentation equation
		corresponding to the initial data $f_0^n$ and $\bar f_0$, respectively. Then from the compactness result in Lemma \ref{lemaa_2} and the continuous dependence on initial data established in Proposition~ \ref{prop_3}, we have
		\[
		f^n(\cdot,T)\to \bar f(\cdot,T)
		\quad \text{strongly in } L^2(\mathbb{R}_+).
		\]
		
		Since the $L^2$-norm is weakly lower semicontinuous, it follows that
		\[
		J(\bar f_0)
		=\frac12\|\bar f(\cdot,T)-f^\ast\|_{L^2(\mathbb{R}_+)}^2
		\le \liminf_{n\to\infty}
		\frac12\|f^n(\cdot,T)-f^\ast\|_{L^2(\mathbb{R}_+)}^2
		= \inf J.
		\]
		
		Hence, $\bar f_0$ is a minimizer of $J$ over $\mathcal{U}_{\mathrm{ad}}$.
	\end{proof}

	\subsection{Derivation of the adjoint equation}
	To derive the adjoint equation of the multiple fragmentation \eqref{0_1} first we define the dual space of the $X_{0,r}$ fro $r\ge 1$
	\[
	X_{\infty,r} := X_{0,r}^\ast = \big\{\varphi:\left(0,\infty\right) \to \mathbb{R} \text{ measurable }: \|\varphi\|_{\infty,r} = \sup_{x>0} \frac{\left|\varphi (x)\right|}{1+x^r} \big \}
	\]
	equipped with the duality pairing
	\[
	\langle f,g \rangle = \int_{\mathbb{R}_+} f(x) g(x) \dd x, \quad f \in X_{0,r},~\text{and}~ g\in  X_{\infty,r}.
	\]
	Since $L^\infty(0,\infty) \subseteq X_{\infty,r}$  and the fragmentation operator $\ds \mathcal{A}$ is weakly dense in $X_{0,r}$, so it is natural to consider $L^\infty(0,\infty)$ as the domain of the adjoint operator.
	\begin{Proposition}
		Let $f \in X_{0,r}$ for some $r\ge 1$, the adjoint fragmentation operator $\ds \mathcal{A}^\ast \varphi(x,t) : \text{dom}(\mathcal{A}^\ast ) := L^\infty(0,\infty) \subseteq X_{\infty,r} \to X_{\infty,r}$ is given by
		\[
		\mathcal{A}^\ast \varphi(x,t)= -\int_{0}^{x} b (y,x) S(x) \varphi(y,t)\dd y + S(x) \varphi(x,t)
		\]
	\end{Proposition}
	\begin{proof}
	In order to compute the adjoint equation, we multiply the equation \eqref{0_1} by $\varphi = \varphi (x,t)$ and integrate over $\ds \mathbb{R}_+ \times \left[0,T\right]$.
	\begin{align}
		\int_{0}^{T} \int_{0}^\infty \varphi \frac{\partial f}{\partial t} \dd x \dd t  = \int_{0}^{T}\int_{0}^{\infty} \varphi (x) \int_{x}^{\infty} b\left(x,y\right)S(y)f(y) \dd y \dd x \dd t- \int_{0}^{T} \int_{0}^{\infty} \varphi(x) S(x) f(x) \dd x \dd t \notag
	\end{align}
	Now integrating by parts in the left hand side and change the order of the integration in the right hand side, we get
	\begin{align*}\allowdisplaybreaks
		\int_{0}^{\infty} \left[\varphi f\right]_0^T \dd x - \int_{0}^{T} \int_{0}^{\infty} f \frac{\partial \varphi}{\partial t}\dd x \dd t
		 =&\int_{0}^{T} \int_{0}^{\infty} \int_{0}^{y} b \left(x,y\right)S(y)\varphi(x)f(y) \dd x \dd y \dd t\\
		 &- \int_{0}^{T} \int_{0}^{\infty}  S(x) \varphi(x)f(x) \dd x \dd t. \notag 
	\end{align*}
	Now by taking Gateaux derivatives with respect to the variable $f$, the above identity reduced to;
	\begin{align}\label{1_2}
		\int_{0}^{T} \int_{0}^{\infty}\delta f \left(- \frac{\partial \varphi}{\partial t} -  \int_{0}^{x} b \left(y,x\right)S(x)\varphi(y) \dd y + S(x) \varphi(x)\right)\dd x \dd t + \int_{0}^{\infty} \left[\varphi \delta f\right]_0^T \dd x=0
	\end{align}
	Now by selecting 
	\begin{align}\label{1_3}
		- \frac{\partial \varphi}{\partial t} -  \int_{0}^{x} b \left(y,x\right)S(x)\varphi(y) \dd y + S(x) \varphi(x)=0,
	\end{align}
	equation \eqref{1_2} becomes
	\begin{align*}
		\int_{0}^{\infty} \left[\varphi (T) \delta f(T) - \varphi(0) \delta f(0)\right] \dd x = 0.
	\end{align*}
	Now assume that $\ds \varphi (T) = f(T) - f^\ast$
	\begin{align}\label{1_4}
		\int_{0}^{\infty} \left(f(T) - f^\ast\right) \delta f \dd x = \int_{0}^{\infty} \varphi(0) \delta f(0) \dd x.
	\end{align}
	This derivation leads us to the adjoint equation
	\begin{align}\label{1_5}
		 \frac{\partial \varphi(x,t)}{\partial t} = \mathcal{A}^\ast \varphi(x,t), \quad\text{with}\quad  \varphi (T) = f(T) - f^\ast.
	\end{align}
	where the adjoint fragmentation operator is given by
	\begin{align*}
	\mathcal{A}^\ast\varphi(x,t) = -\int_{0}^{x} b (y,x) S(x) \varphi(y,t)\dd y + S(x) \varphi(x,t).	
	\end{align*}
	\end{proof}
	
	This equation quantifies the sensitivity of the solution to perturbations in the initial condition. It is important to note that the adjoint equation is solved backward in time, from \( t = T \) to \( t = 0 \). Moreover, system~\eqref{1_5} is well posed precisely when the original forward problem is well-posed.

\subsection{Gradient descent method}
	In order to calculate the gradient $\ds \frac{\delta J}{\delta f_0}$, we take the first variation on the cost functional with respect to $f_0$ in \eqref{0_2} 
	\begin{align*}
		\delta J (f_0) = \int_{0}^{\infty} \delta f (T) \left(f(T) - f^\ast\right)\dd x.
	\end{align*}
	Using the relation \eqref{1_4} on the above expression
	\begin{align}\label{2_1}
		\delta J(f_0) = \int_{0}^{\infty} \varphi(0) \delta f(0) \dd x.
	\end{align}
	To minimize the cost functional $J$, we will apply the classical gradient descent algorithm with some initialization of $f_0$. The iteration can be written in the following form
	\begin{align}\label{2_2}
		 f_0^{k+1} = f_0^k - \varepsilon_0 \nabla J^k\quad\text{for}\quad k
		 \ge 0,
	\end{align}
	and the gradient $\ds  \nabla J^k = \frac{\delta J \left(f_0^k\right)}{\delta f_0}$. Then from the relation \eqref{2_1}, the iteration scheme reduced into the following form:
	\begin{align}\label{2_3}
		f_0^{k+1} = f_0^k - \varepsilon_0 \varphi^k(0)\quad\text{for}\quad k
		\ge 0.
	\end{align}
	Here, more sophisticated or optimal strategies for selecting \( \varepsilon_0 \) are possible, However, for simplicity, we choose the step size \( \varepsilon_0 \) to be constant and independent of \( k \). 
	
	\section{Numerical approach to the inverse design problem}\label{S_4}
	To design the numerical scheme for the both state and adjoint equation, we need to truncate the computational domain as $\ds \Omega:= \left(0, \mathcal{R}\right]$. The corresponding truncated state equation 
	\begin{align}\label{3_0}
		\frac{\partial f(x,t)}{\partial t}:=\int_{x}^{\mathcal{R}} b \left(x,y\right)S(y)f(y,t) \dd y - S(x) f(x,t),
	\end{align}
	We discretize the computational domain \( \Omega \) into \( I \) subintervals 
	\( \Lambda_i := [x_{i-1/2}, x_{i+1/2}] \) for \( i = 1,2,\cdots,I \). 
	The representative point \( x_i \) of each cell \( \Lambda_i \) is defined by $x_i := \frac{x_{i-1/2} + x_{i+1/2}}{2}$,
	with \( x_{-1/2} = 0 \) and \( x_{I+1/2} = \mathcal{R} \). 
	Furthermore, the step size of each subinterval \( \Lambda_i \) is given by $\Delta x_i := x_{i+1/2} - x_{i-1/2} \le \Delta x$,
	for each \( i \in \{1,2,\cdots,I\} \).
	\par Moreover, for a fully discrete formulation, the time domain must also be discretized. 
	To this end, we partition the interval \( [0,T] \) into \( N \) uniform subintervals 
	\( \tau_n := [t_n, t_{n+1}) \) for \( n = 0,1,\cdots,N-1 \), 
	where \( t_n = n \Delta t \) and \( N \Delta t = T \).
	\subsection{Finite volume scheme for forward fragmentation problem (FVS)}
	For the finite volume scheme we need to consider the average value of the solution $f(x,t)$ in each of the subinterval $\ds \Lambda_i:= \left[x_{i-1/2}, x_{i+1/2}\right]$ as follows:
	\begin{align*}
		f_i^n = \frac{1}{\Delta x_i} \int_{\Lambda_i} f(x,t^n) \dd x
	\end{align*}
	To obtain the numerical scheme we integrate equation \eqref{3_0} over $\ds \Lambda_i \times [t_n, t_{n+1}]$ and get the following scheme:
	\begin{align}\label{3_1}
		f_i^{n+1} = f_i^n + \frac{\Delta t}{\Delta x_i} \left(\B_i^n -\D_i^n\right).
	\end{align}
	With the initial condition
	\begin{align}\label{3_2}
		f_i(0) = f_{0,i}\quad i=1,2,\cdot\cdot\cdot.
	\end{align}
	Where birth flux $\B_i^n$ can be computed as follows:
	\begin{align*}
		\B_i^n &= \frac{1}{\Delta x_i}\int_{x_{i-1/2}}^{x_{i+1/2}}\int_{x}^{x_{I+1/2}} b \left(x,y\right)S(y)f(y,t^n) \dd y \dd x \\
		& = \frac{1}{\Delta x_i} \left[\int_{x_{i-1/2}}^{x_{i+1/2}}\int_{x_{i-1/2}}^{y} b \left(x,y\right)S(y)f(y,t^n) \dd x\dd y + \sum_{j=i+1}^{I}\int_{x_{j-1/2}}^{x_{j+1/2}}\int_{x_{i-1/2}}^{x_{i+1/2}}b \left(x,y\right)S(y)f(y,t^n) \dd x\dd y \right] \\
		& = \frac{1}{\Delta x_i} \sum_{j=i}^{I} S_j f_j^n \Delta x_j  \int_{x_{i-1/2}}^{p_j^i} b(x,x_j)\dd x.
	\end{align*}
	Where
	\begin{align*}
		p^i_j = \left\{\begin{array}{ll}
			x_{i+1/2}, &\mbox{for}\quad j\ne i, \vspace{0.2cm}\\
			x_i, &\mbox{otherwise}.
		\end{array}\right.
	\end{align*}
	And the death flux $\D_i^n$ is given by
	\begin{align*}
		\D_i^n = S_i f_i^n.
	\end{align*}
	Using a derivation similar to that in \cite{kumar2015development}, one can show that the above numerical scheme is convergent and achieves second–order accuracy. However, from our computations we observe that although the above scheme is capable of 
	predicting the total number of particles, it fails to conserve the total mass over any given time interval. To address this issue, the following weighted finite volume scheme was proposed in \cite{kumar2015development}.
	
	\subsection{Weighted finite volume scheme for forward equation (WFVS)}
	Mass conservation and total number or particles prediction is the most essential property of any numerical schemes for a particulate process. In this regards, we also calculate the forward fragmentation equation through the following weighted finite volume scheme \cite{kumar2015development}
	  \begin{align}
	  	f_i^{n+1} = f_i^n + \frac{\Delta t}{\Delta x_i} \left(\tilde{\B}_i^n -\tilde{\D}_i^n\right).
	  \end{align}
	Where
	\begin{align*}
		\tilde{\B}_i^n := \frac{1}{\Delta x_i} \sum_{j=i}^{I} w^b_j S_j f_j^n \Delta x_j \int_{x_{i-1/2}}^{p_j^i} b(x,x_j)\dd x \quad \text{and}\quad \tilde{\D}_i^n := w_i^d S_i f_i^n.
	\end{align*}
	The wight functions $\ds w^b$ and $w^d$ can be defined as
	\begin{align*}
		w^b_j = \frac{x_j\left[\nu\left(x_j\right) - 1\right]}{\sum_{i=1}^{j-1} \left(x_j - x_i\right) \B_{i,j}} \qquad\text{and}\quad 	w^d_i = \frac{w^b_i}{x_i} \sum_{j=1}^{i} x_j \B_{i,j}, \qquad j=2,3,\cdot\cdot\cdot ,I.
	\end{align*}
	The values of $w^b_1$ and $w_1^d$ are considered as zero. Similar to the FVS, the WFVS is also second–order accurate.
	\subsection{New developement of the adjoint numerical scheme}
	To construct a numerical scheme for the adjoint equation, we adopt the \emph{continuous adjoint} approach. Specifically, the adjoint equation \eqref{1_5} is first derived at the continuous level and then discretized using the same finite volume framework as that employed for the forward problem. Since the adjoint equation is solved backward in time, we integrate \eqref{1_5} over the space--time control volume 
	\( \Lambda_i \times [t_{n+1}, t_n] \), which leads to the following discrete adjoint scheme:
	\begin{align}\label{4_1}
		\varphi_i^n = \varphi_i^{n+1} +  \frac{\Delta t}{\Delta x_i} \left(\hat{\B}_i^{n+1} -\hat{\D}_i^{n+1}\right).
	\end{align}
	with the initial condition
	\begin{align}\label{4_2}
		\varphi_i (T) = f_i(T) - f^\ast_i, \quad i=1,2,\cdot\cdot\cdot, I.
	\end{align}
	Where adjoint birth flux $\hat{\B}_i^{n+1}$ can be computed as follows:
	\begin{align*}
		\hat{\B}_i^{n+1} &= \frac{1}{\Delta x_i} \int_{x_{i-1/2}}^{x_{i+1/2}}\int_0^{x} b \left(y,x\right)S(x)\varphi(y,t^{n+1}) \dd y \dd x \\
		& = \frac{1}{\Delta x_i} \left[\int_{x_{i-1/2}}^{x_{i+1/2}}\int_y^{x_{i+1/2}} b \left(y,x\right)S(x)\varphi(y,t^{n+1}) \dd x\dd y + \sum_{j=1}^{i-1}\int_{x_{j-1/2}}^{x_{j+1/2}}\int_{x_{i-1/2}}^{x_{i+1/2}}b \left(y,x\right)S(x)\varphi(y,t^{n+1}) \dd x\dd y \right] \\
		& = \frac{1}{\Delta x_i} \sum_{j=1}^{i}  \varphi_j^{n+1} \Delta x_j \int_{q_j^i}^{x_{i+1/2}}S(x) b(x_j,x)\dd x.
	\end{align*}
	Where
	\begin{align*}
		q^i_j = \left\{\begin{array}{ll}
			x_{i-1/2}, &\mbox{for}\quad j\ne i, \vspace{0.2cm}\\
			x_i, &\mbox{otherwise}.
		\end{array}\right.
	\end{align*}
	And the death flux $\hat{\D}_i^n$ is given by
	\begin{align*}
		\hat{\D}_i^{n+1} = S_i \varphi_i^{n+1}.
	\end{align*}
	
	It is worth noting that, in optimal control algorithms, the adjoint system is often constructed using the so-called \emph{discrete adjoint} approach, in which the adjoint equations are derived directly from the discretized forward model. In the present study, however, the forward discretization—particularly the weighted finite volume scheme (WFVS)—relies on weighted reconstructions that render the derivation of an exact discrete adjoint formulation technically involved. By contrast, the discretization of the continuous adjoint equation remains straightforward and computationally efficient. 
	
	\section{Numerical Tests}\label{S_5}
	For the numerical test, we consider a benchmark problem for which the exact solution is already known for a given initial data. We consider the computation domain $\ds \Omega=[0,5]$ and we divide the computational domain into the non uniform subintervals using the geometric recurrence relation $x_{i+1/2} = r x_{i-1/2}$, where $r = 1.4$. For the numerical experiment we take the final time $T=2s$ and divide the time domain $[0,T]$ into $20$ uniform sub intervals. For each of the test problem we consider the guess of initial data $\ds f(x,0) = x \exp[-100x^{100}]$ in the first iteration.  Our goal is to build a numerical technique which is capable to predict the target exact solution $\ds f^\ast$ as well as the corresponding initial data  $\ds f_0$ by means of a numerical version of the optimization algorithm above.
	
	In each of the numerical experiment we consider both the FVS and WFVS scheme for the forward computation. And we will see the to achieve the target function WVS gives better result with less number of iteration in comparison to the FVS. In order to compute the error of tolerance, we calculate the absolute difference between the exact and numerically approximated function for the target $\ds f(x,T) - f^\ast(x)$ and the initial datum $\ds f(x,0) -  f_0(x)$. The explicit formula for the $L^\infty$ error for approximated data with respect to the target as well as the initial data is given  as follows:
	\begin{align*}
		& E(f^\ast) := \max_i \left|f(x,T) - f^\ast(x)\right| \\
		& E(f_0) := \max_i \left|f(x,0) -  f_0(x)\right|.
	\end{align*}   
	
	\subsection{Test Case: I (Linear fragmentation growth)}
	For the first numerical experiment we consider the test problem with linear fragmentation rate. In this regard, we consider the fragmentation rate have the linear growth rate $S(x)=x$ with the the daughter distribution function with `power-law'; rate: $b(x,y)= \frac{2}{y}$. With this above setup the benchmark problem have the exact solution in the following explicit form:
	\begin{align}\label{5_1}
		f(x,t) = \left(1+ t/s\right)^2 e^{-x\left(s+t\right)}
	\end{align}  
	corresponding to the initial data 
	\begin{align}\label{5_2}
		 f_0 (x) = e^{-sx}.
	\end{align}
   The availability of the exact solution at any specific time $t=T$, allow us to formulate the inverse design of the considered problem. We consider the given exact solution at the final time $t =T$ as a target function at time $t=T$ with  $s=1$. Moreover, with respect to the  target \eqref{5_1} we aim to recover the approximate initial datum \eqref{5_2} by the computational version of the gradient-adjoint methodology described above.
   
     For both the FVS and WFVS we divide the computational domain into $35$ non uniform grid points generated by the way mentioned above. Also, the maximum number of iteration for the adjoint based gradient method is set to $\ds 50$ with the learning rate $\varepsilon_0=0.002$. 
     \subsubsection{\textbf{Gradient Verification:}} We implement the Taylor test to validate the correctness of the adjoint-based gradient derived in Section~\ref{S_3}. 
     To this end, we consider the following discrete objective function:
     \[
     J_{\Delta x}\left(f_0\right) := \frac{1}{2} \sum_{i=1}^I 
     \left| f_i^{\mathrm{approx}} - f_i^{\mathrm{Exact}} \right|^2 \Delta x_i.
     \]
     Corresponding to this discrete objective functional, the remainder term is defined by
     \begin{align}
     	R_{\Delta x,d} \left(\eta\right)
     	:= \left| 
     	J_{\Delta x}\left(f_0 + \eta d\right) 
     	- J_{\Delta x}\left(f_0\right) 
     	- \eta \left\langle \nabla J_{\Delta x}\left(f_0\right), d \right\rangle_{\Delta x}
     	\right|.
     \end{align}
     Here, the discrete inner product is defined as
     \[
     \left\langle \theta , \psi \right\rangle_{\Delta x}
     := \sum_{k} \theta_k \psi_k \Delta x_k,
     \]
     where $\eta$ denotes the perturbation size applied to the initial guess $f_0$, 
     and $d$ is a fixed perturbation direction.
     
     \begin{table}[h]
     	\centering
     	\Large \renewcommand{\arraystretch}{1.0}
     	\begin{tabular}{|c|c|c|}
     		\hline
     		$\mathbb{\eta}$ & $R_{\Delta x,d}$ & $R_{\Delta x,d}/\eta^2$ \\ 
     		\hline
     		$1.00E-01$  & $6.565E-01$  & $6.565E+01$  \\ 
     		\hline
     		$1.00E-02$  & $5.157E-03$  & $5.157E+01$  \\ 
     		\hline
     		$1.00E-03$  & $4.922E-05$ & $4.922E+01$  \\ 
     		\hline
     	\end{tabular}
     	\caption{Taylor reminders with respect to different perturbation size.}
     	\label{tab_taylor}
     \end{table} 
     To compute the above remainder term, we consider the same test case and choose a unit direction $d$, normalized with respect to the discrete $L^2$ inner product. Table~\ref{tab_taylor} shows that the Taylor remainder remains consistent as $\eta \downarrow 0$. Moreover, the remainder satisfies
     \( \ds
     R_{\Delta x,d} = \mathcal{O}\left(\eta^2\right),
     \)
     which confirms that the proposed adjoint-based gradient has been implemented correctly.

     In Figures \ref{f1} and \ref{f2}, we plot the numerical results for the target and initial datum with the exact target and initial datum computed by both forward schemes. After $50$ iterations, it can be observed that the approximate and exact target data as well as the initial datum matches perfectly. 
   
	\begin{figure}[htp]
		\begin{subfigure}{.45\textwidth}
			\centering
			\includegraphics[width=1.0\textwidth]{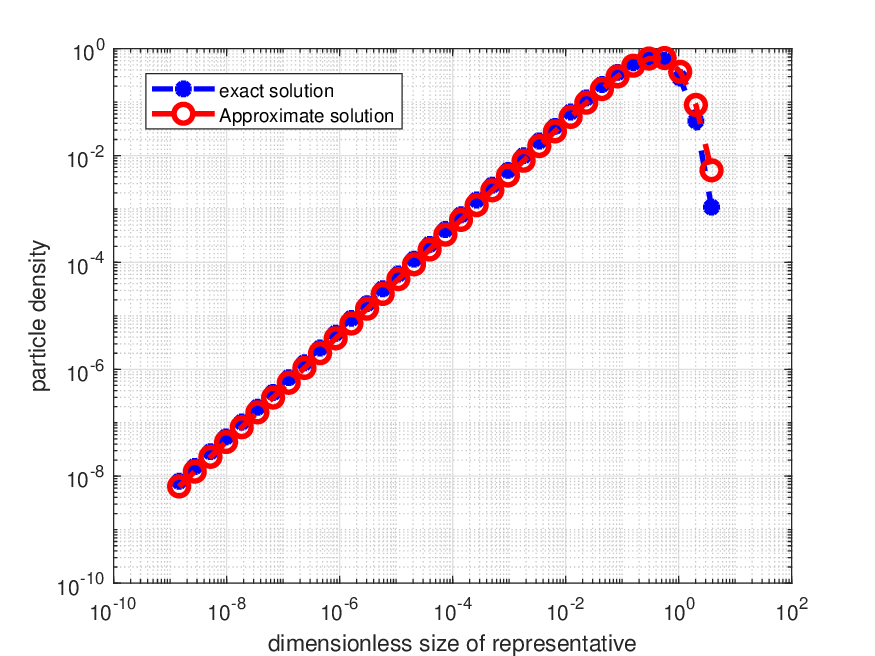}
			\caption{FVS}
			\label{f1_1}
		\end{subfigure}
		\begin{subfigure}{.45\textwidth}
			\centering
			\includegraphics[width=1.0\textwidth]{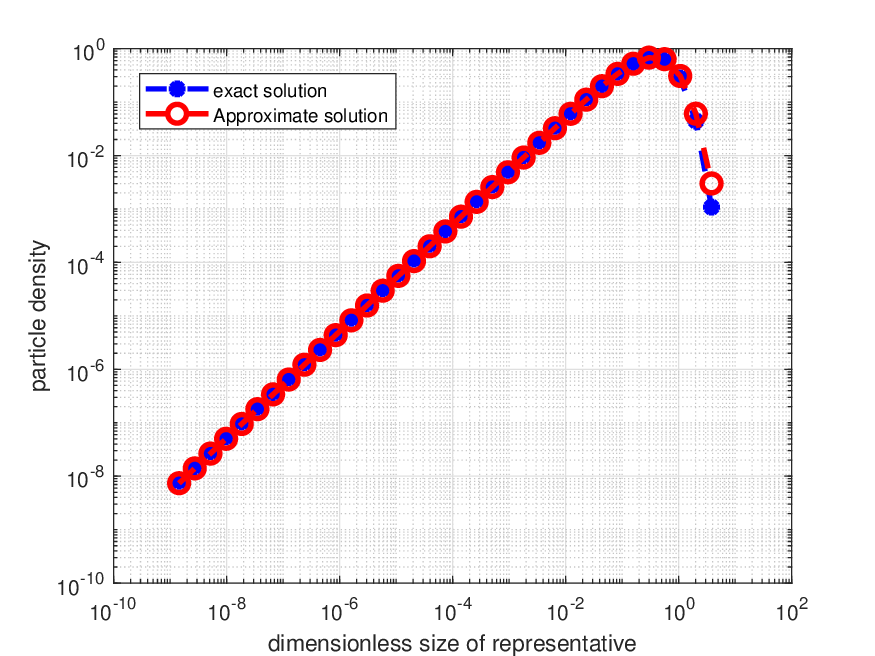}
			\caption{WFVS}
			\label{f1_2}
		\end{subfigure}
		\caption{Comparison between exact and approximated target function by  FVS \& WFVS in linear scale.}
		\label{f1}
	\end{figure}
	 
	\begin{figure}[htp]
		\begin{subfigure}{.45\textwidth}
			\centering
			\includegraphics[width=1.0\textwidth]{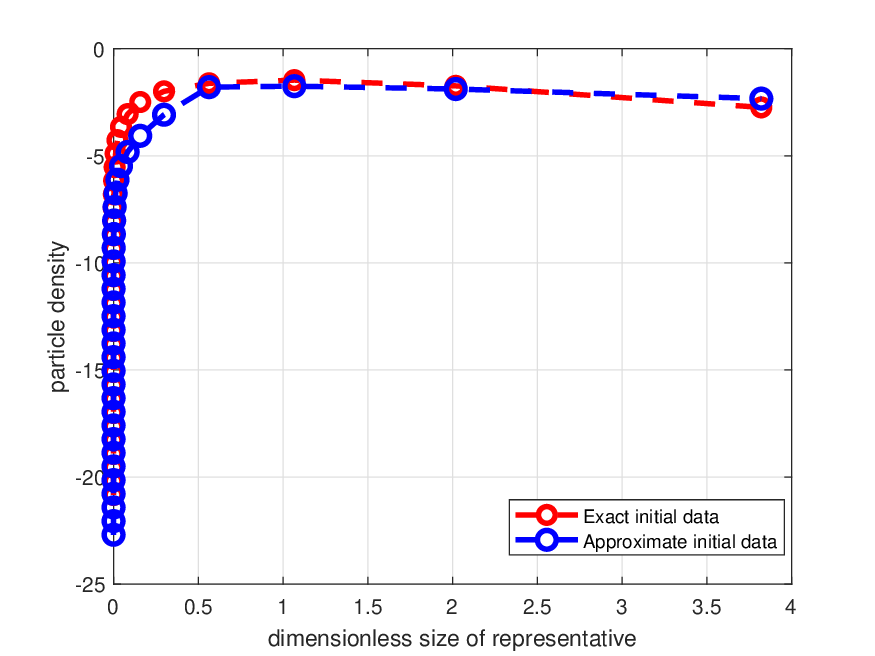}
			\caption{FVS}
			\label{f2_1}
		\end{subfigure}
		\begin{subfigure}{.45\textwidth}
			\centering
			\includegraphics[width=1.0\textwidth]{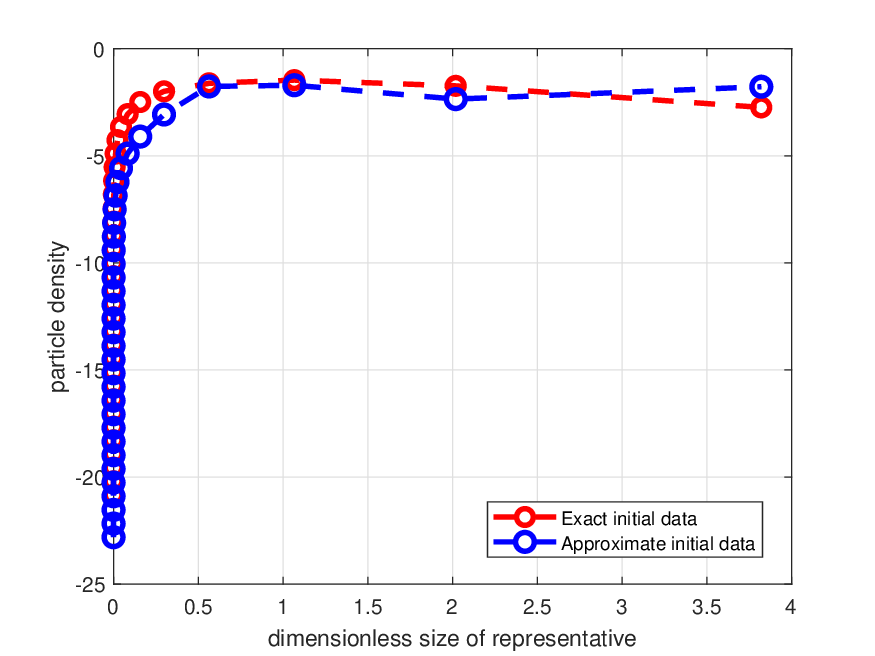}
			\caption{WFVS}
			\label{f2_2}
		\end{subfigure}
		\caption{Comparison between exact and approximated initial datum by  FVS \& WFVS in logarithmic scale.}
		\label{f2}
	\end{figure}
	\begin{table}[h]
		\centering
		\Large \renewcommand{\arraystretch}{1.8} 
		 \begin{tabular}{|c|c|c|c|}
		\hline
			\textbf{Schemes} \qquad& $\mathbf{E(f^\ast)}$ & $\mathbf{E(f_0)}$ & \textbf{Iterations} \\
			\hline 
			FVS & $8.71e-02$ & $9.07e-02$ & $50$ \\
			\hline
			WFVS & $3.32e-02$ & $1.049e-01$ & $50$ \\
			\hline
		\end{tabular}
		\caption{ Error with number of iterations for the Test case I}
		\label{tab:1}
	\end{table}
	Table~\ref{tab:1} presents the qualitative errors from the numerical experiments for both forward schemes. It can be observed that, in approximating the target function, the forward WFVS achieves higher accuracy than the FVS. However, after the same number of iterations, the error \( E(f_0) \) for the forward FVS is smaller than that for the forward WFVS.

	\begin{figure}[htp]
		\begin{subfigure}{.45\textwidth}
			\centering
			\includegraphics[width=1.0\textwidth]{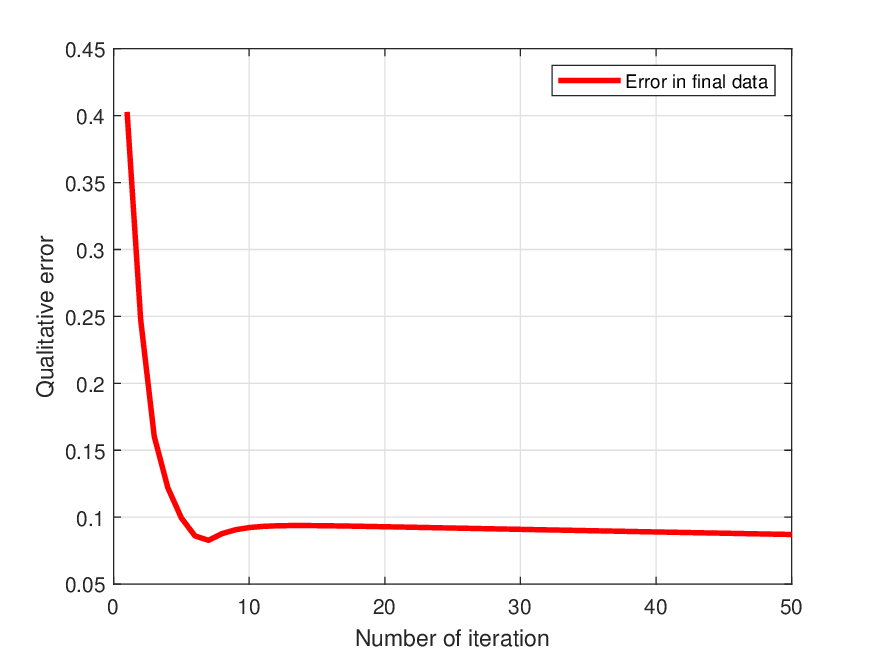}
			\caption{$\ds E\left(f^\ast\right)$}
			\label{f3_1}
		\end{subfigure}
		\begin{subfigure}{.45\textwidth}
			\centering
			\includegraphics[width=1.0\textwidth]{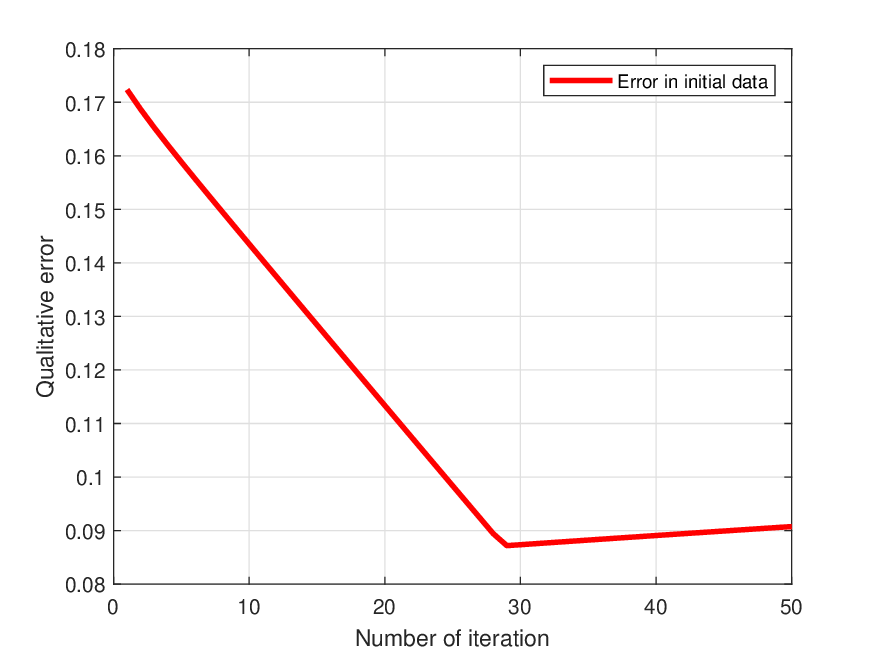}
			\caption{$\ds E\left(f_0\right)$}
			\label{f3_2}
		\end{subfigure}
		\caption{Change of target error and initial datum error by  FVS.}
		\label{f3}
	\end{figure}
	We also plot the change in error with respect to the number of iterations in Figures \ref{f3} and \ref{f4}, corresponding to the target function $\ds f(x,T) - f^\ast(x)$ and the initial datum $\ds f(x,0) - f_0(x)$, respectively. Figures \ref{f3} and \ref{f4} show that the qualitative error for both the target function and the initial datum decreases as the number of iterations increases. 
	\begin{figure}[htp]
		\begin{subfigure}{.45\textwidth}
			\centering
			\includegraphics[width=1.0\textwidth]{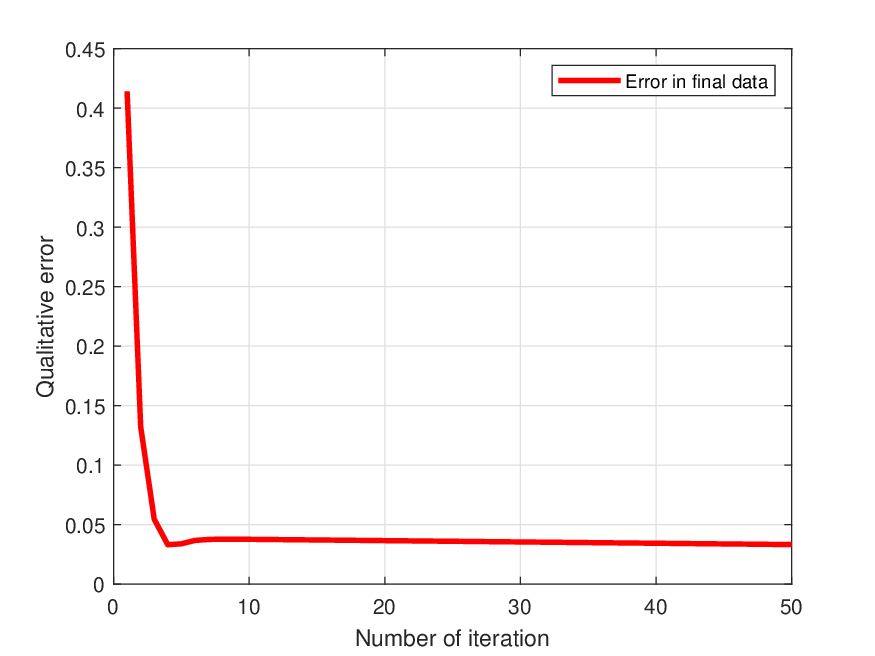}
			\caption{$\ds E\left(f^\ast\right)$}
			\label{f4_1}
		\end{subfigure}
		\begin{subfigure}{.45\textwidth}
			\centering
			\includegraphics[width=1.0\textwidth]{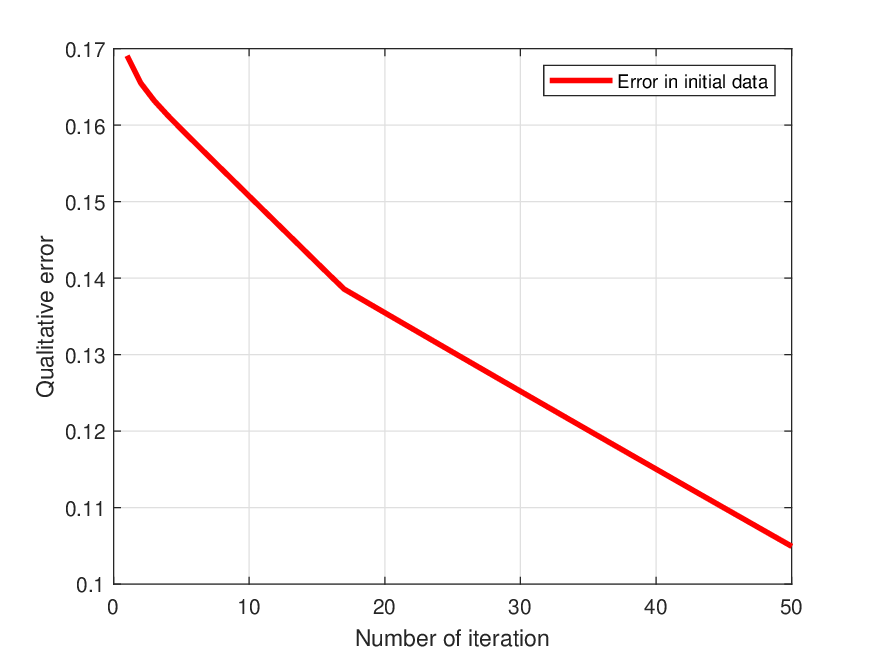}
			\caption{$\ds E\left(f_0\right)$}
			\label{f4_2}
		\end{subfigure}
		\caption{Change of target error and initial datum error by WFVS.}
		\label{f4}
	\end{figure}
	
	\subsection{Test Case: II (Nonlinear fragmentation growth)}
	In the second test problem we consider the fragmentation rate having quadratic growth rate $S(x) = x^2$ with the same daughter distribution function as Test case I. With this setup, the benchmark problem have the exact solution in the following explicit form:
	\begin{align}\label{5_3}
		f(x,t) = \left(1+ 2t\left(1+x\right)\right) e^{-tx^2-x}
	\end{align}  
	Corresponding to the same initial data \eqref{5_2}, and similar to the previous test case, we implement the above gradient-based algorithm by employing both the FVS and WFVS schemes to approximate the target exact solution $\ds f^\ast$ as well as the corresponding initial data $\ds f_0$. In this regard, we divide the computational domain into $25$ non-uniform grid points and choose the learning rate $\ds \varepsilon_0 = 0.0015$.
	
	\begin{figure}[htp]
		\begin{subfigure}{.45\textwidth}
			\centering
			\includegraphics[width=1.0\textwidth]{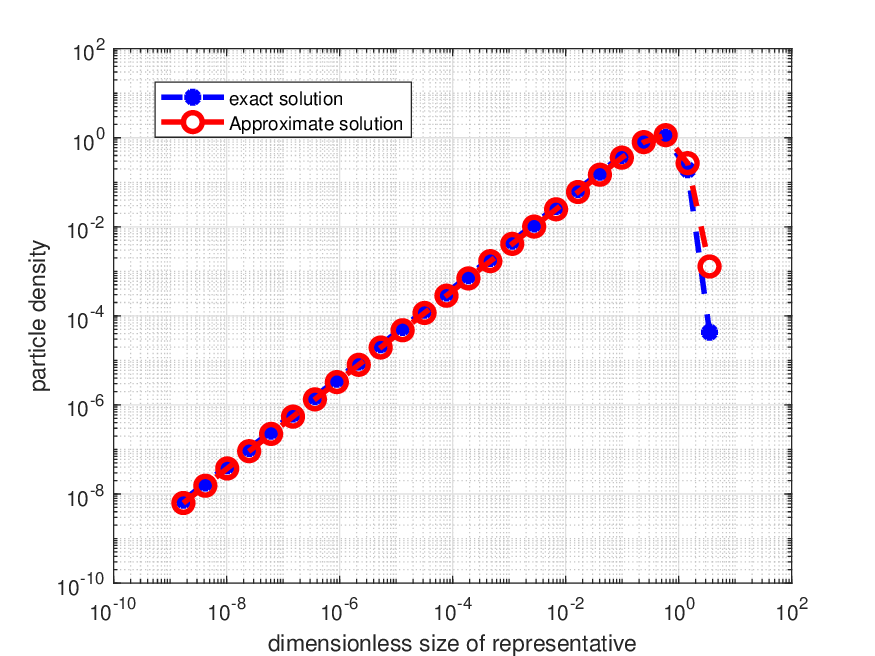}
			\caption{FVS}
			\label{f5_1}
		\end{subfigure}
		\begin{subfigure}{.45\textwidth}
			\centering
			\includegraphics[width=1.0\textwidth]{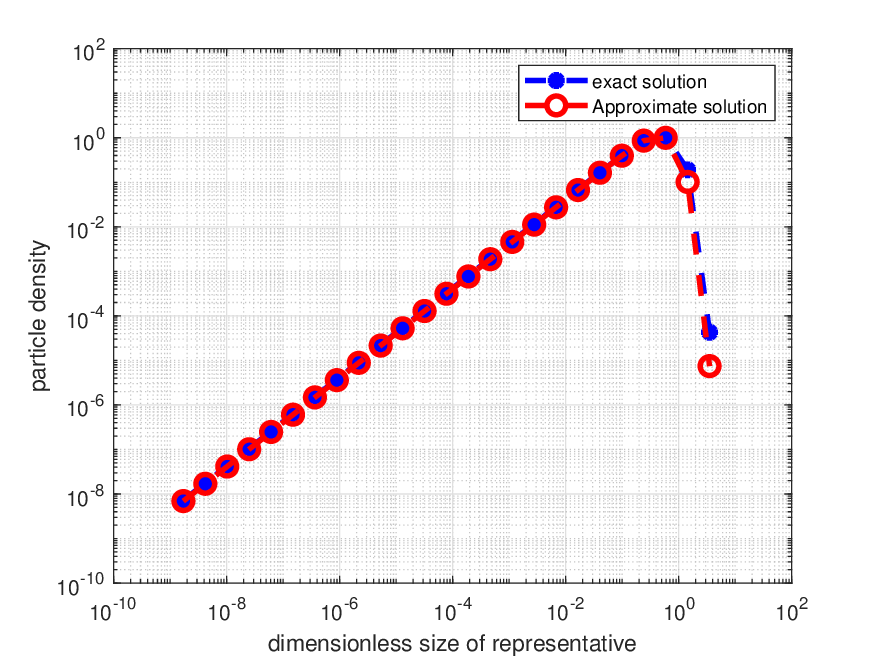}
			\caption{WFVS}
			\label{f5_2}
		\end{subfigure}
		\caption{Comparison between exact and approximated target function by  FVS \& WFVS.}
		\label{f5}
	\end{figure}
	
	\begin{figure}[htp]
		\begin{subfigure}{.45\textwidth}
			\centering
			\includegraphics[width=1.0\textwidth]{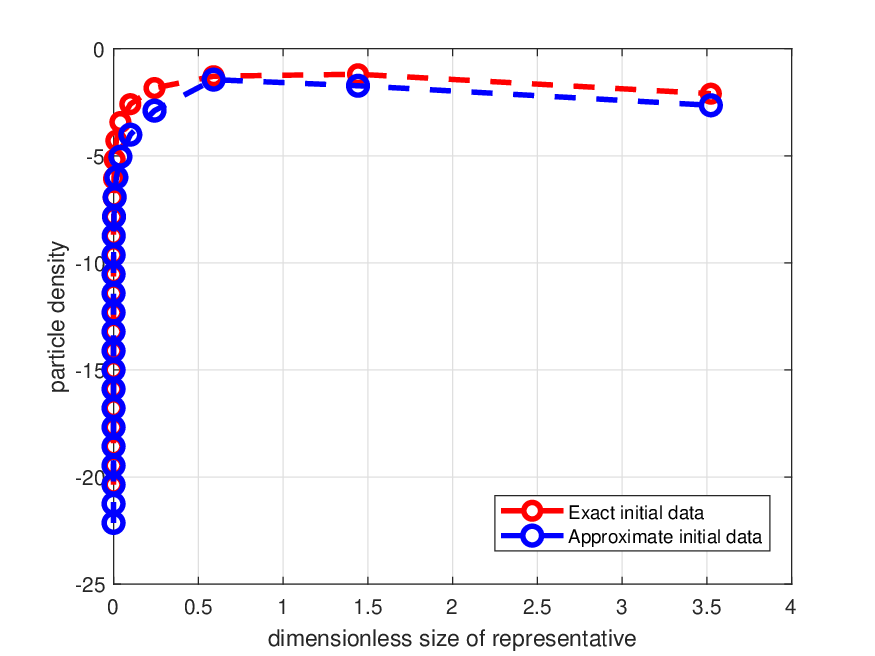}
			\caption{FVS}
			\label{f6_1}
		\end{subfigure}
		\begin{subfigure}{.45\textwidth}
			\centering
			\includegraphics[width=1.0\textwidth]{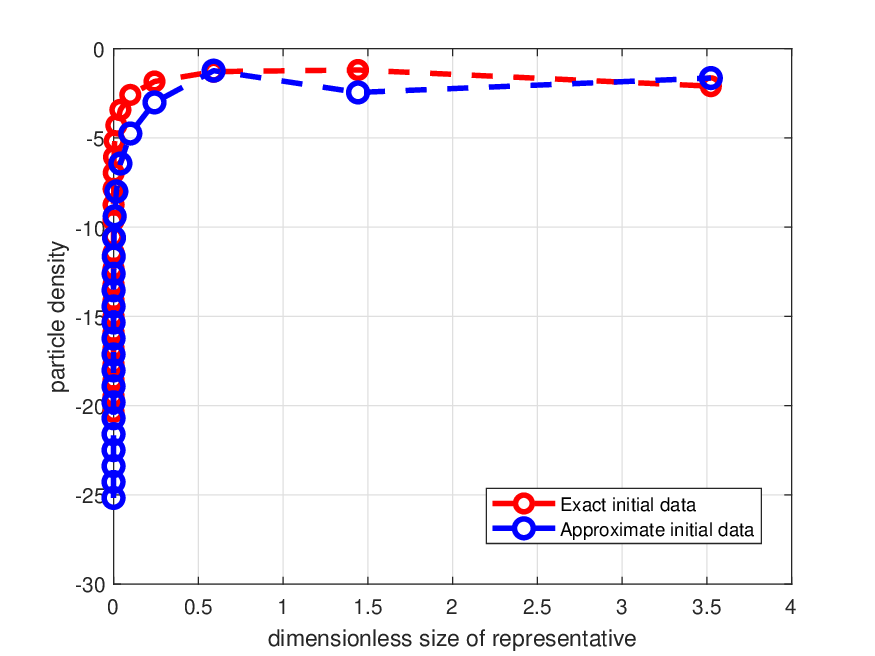}
			\caption{WFVS}
			\label{f6_2}
		\end{subfigure}
		\caption{Comparison between exact and approximated initial datum by  FVS \& WFVS in logarithmic scale.}
		\label{f6}
	\end{figure}
	
		\begin{table}[h]
		\centering
		\Large \renewcommand{\arraystretch}{1.8} 
		\begin{tabular}{|c|c|c|c|}
			\hline
			\textbf{Schemes} \qquad& $\mathbf{E(f^\ast)}$ & $\mathbf{E(f_0)}$ & \textbf{Iterations} \\
			\hline 
			FVS & $1.043e-01$ & $1.262e-01$ & $150$ \\
			\hline
			WFVS & $8.86e-02$ & $2.172e-01$ & $15$ \\
			\hline
		\end{tabular}
		\caption{ Error with number of iterations for the Test case I}
		\label{tab:2}
	\end{table}
	
	Table \ref{tab:2} presents the qualitative error for the numerical experiment using both forward schemes.Since the WFVS performs more effectively than the FVS in capturing the final state in a purely forward test \cite{kumar2015development}, it is observed that the WFVS attains a smaller error in the target function within significantly fewer iterations compared to the FVS. However, when approximating the initial data, the FVS proves to be more effective, as the error $E(f_0)$ decreases within just a few iterations.
	
	\begin{figure}[htp]
		\begin{subfigure}{.45\textwidth}
			\centering
			\includegraphics[width=1.0\textwidth]{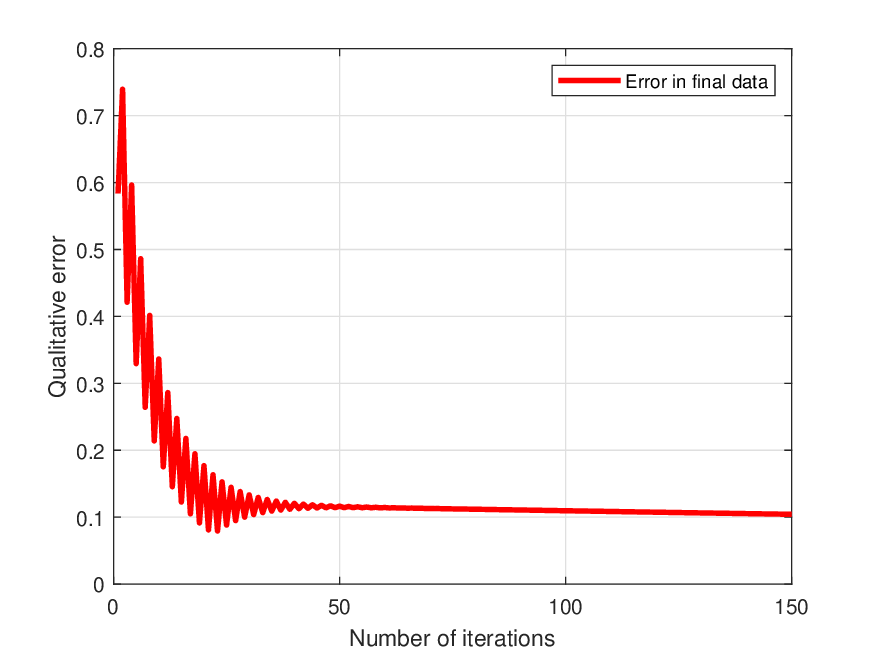}
			\caption{$\ds E\left(f^\ast\right)$}
			\label{f7_1}
		\end{subfigure}
		\begin{subfigure}{.45\textwidth}
			\centering
			\includegraphics[width=1.0\textwidth]{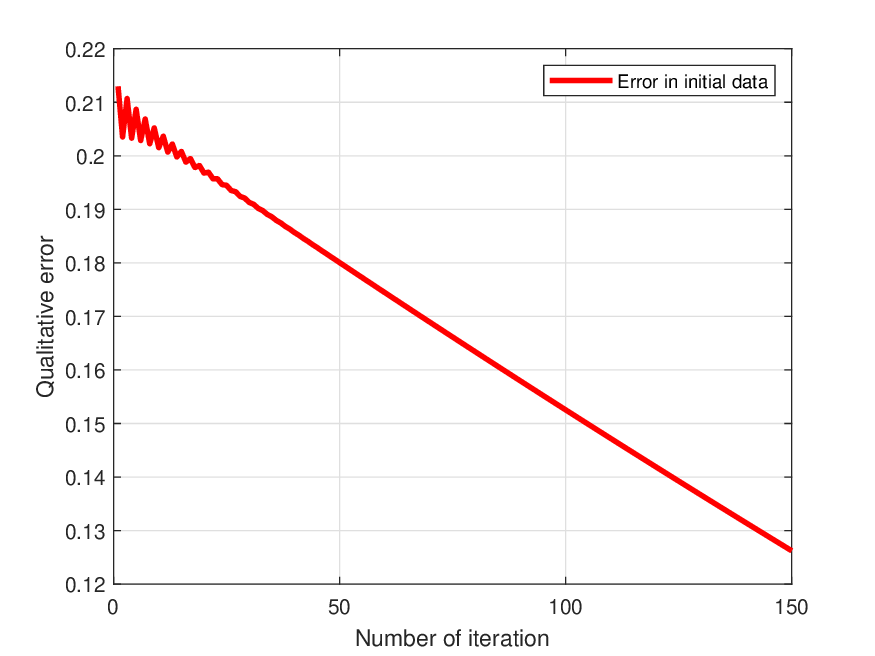}
			\caption{$\ds E\left(f_0\right)$}
			\label{f7_2}
		\end{subfigure}
		\caption{Change of target error and initial datum error by  FVS.}
		\label{f7}
	\end{figure}
	
	We also plot the change in error with respect to the number of iterations in Figures \ref{f7} and \ref{f8}, corresponding to the target function $\ds f(x,T) - f^\ast(x)$ and the initial condition $\ds f(x,0) - f_0(x)$. Figure \ref{f7} shows that the qualitative error obtained using the FVS decreases with the number of iterations for both the target function and the initial data. However, in the case of the forward WFVS, while the error $E(f^\ast)$ decreases over iterations, the error $E(f_0)$ increases. 
	\begin{figure}[htp]
		\begin{subfigure}{.45\textwidth}
			\centering
			\includegraphics[width=1.0\textwidth]{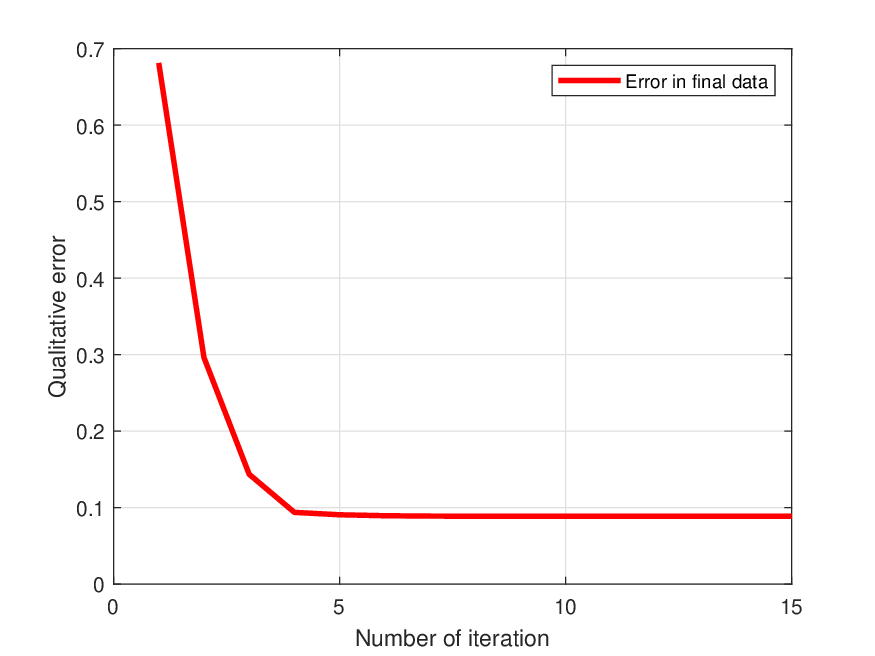}
			\caption{$\ds E\left(f^\ast\right)$}
			\label{f8_1}
		\end{subfigure}
		\begin{subfigure}{.45\textwidth}
			\centering
			\includegraphics[width=1.0\textwidth]{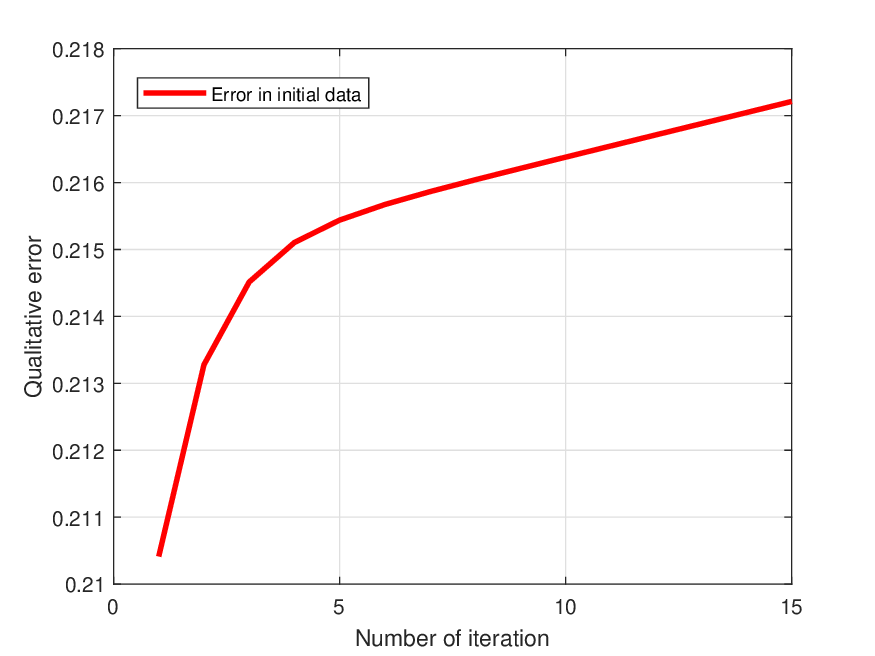}
			\caption{$\ds E\left(f_0\right)$}
			\label{f8_2}
		\end{subfigure}
		\caption{Change of target error and initial datum error by WFVS.}
		\label{f8}
	\end{figure}
	
	Although the weighted finite volume scheme yields a smaller terminal deviation, the reconstructed initial data may exhibit a larger mismatch from the exact profile. This behavior is not contradictory but rather reflects the intrinsic ill-posedness of the backward reconstruction problem. The forward solution operator is smoothing, whereas its inverse is unstable; consequently, small discrepancies at the final time may correspond to significant variations in the initial state. 
	
	Moreover, the weighted discretization introduces additional numerical diffusion, which enhances stability and suppresses oscillations in the forward evolution, but may attenuate high-frequency components that are essential for accurate recovery of the initial condition. Therefore, improved accuracy at the terminal time does not necessarily imply a more accurate reconstruction of the initial data.
	
	 \section*{Acknowledgments}
	This work was initiated during the author’s stay at the Chair for Dynamics, Control, Machine Learning and Numerics at Friedrich–Alexander University Erlangen–Nürnberg, Germany. The author gratefully acknowledges the hospitality of the Chair during this period. 
	
	The author also expresses sincere gratitude to Professor Enrique Zuazua for his invaluable mentorship, insightful discussions, and continuous guidance throughout the development of this work. His expertise and encouragement have greatly contributed to the direction and quality of this research.
	\bibliographystyle{unsrt}
	
	\bibliography{existence}
\end{document}